\documentclass[
final
]{dmtcs-episciences}

\usepackage[utf8]{inputenc}
\usepackage{subfigure}

\usepackage[round]{natbib}

\usepackage[justification=centering]{caption}

\newtheorem{thm}{Theorem}[section]
\newtheorem{lem}[thm]{Lemma}
\newtheorem{claim}[thm]{Claim}
\newtheorem{proposition}[thm]{Proposition}

\newtheorem{remark}[thm]{Remark}
\newtheorem{defi}[thm]{Definition}

\newcommand{\cF}{\mathcal{F}}
\newcommand{\cD}{\mathcal{D}}
\newcommand{\cS}{\mathcal{S}}

\author{Jovana Forcan\affiliationmark{1,2}
  \and Mirjana Mikalački\affiliationmark{1}\thanks{The research is partly supported by Provincial Government for Higher Education and Scientific Research, Republic of Serbia, Grant No.\ 142-451-3227/2020 and by Ministry of Education, Science and Technological Development, Republic of Serbia, Grant No.\ 451-03-68/2022-14/200125.}}
\title[Maker--Breaker total domination game on cubic graphs]{Maker--Breaker total domination game on cubic graphs}
\affiliation{
  University of Novi Sad, Faculty of Sciences, Department of Mathematics and Informatics, Serbia. \\
  University of East Sarajevo, Faculty of Philosophy, Bosnia and Herzegovina.}
\received{2021-09-28}
\revised{2022-04-18}
\accepted{2022-04-26}
\begin{document}
\publicationdetails{24}{2022}{1}{20}{8529}
\maketitle
\begin{abstract}
We study Maker--Breaker total domination game played by two players, Dominator and Staller, on the connected cubic graphs. Staller (playing the role of Maker) wins if she manages to claim an open neighbourhood of a vertex. Dominator wins otherwise (i.e.\ if he can claim a total dominating set of a graph). For certain graphs on $n\geq 6$ vertices, we give the characterization on those which are Dominator's win and those which are Staller's win. 
\end{abstract}

\section{Introduction}
In a Maker--Breaker game $(X, \cF)$, two players, called \emph{Maker} and \emph{Breaker}, take turns in claiming previously unclaimed element of the \emph{board} $X$, with one of them going first. Game ends when all elements of $X$ are claimed. Set $\cF \subseteq 2^X$ is called the family of winning sets. Maker wins the game $(X, \cF)$ if she claims all the elements of one $F \in \cF$ before the end of the game. Breaker wins otherwise. In this type of games, being the first player is always an advantage, and if the player wins the game as second player then he/she wins as a first player as well. We say that the game is a \emph{Maker's win}, if she has a strategy to win against every strategy of Breaker. Similarly, a \emph{Breaker's win} is defined.

Maker--Breaker games are well studied type of \emph{positional games}, combinatorial, perfect information games, with no chance moves. The pair $(X,\cF)$ is also known as the \emph{hypergraph of the game}. We refer the interested reader to the book of~\cite{BeckBook08} and the recent monograph of~\cite{HKSSBook14} for thorough analysis of various types of positional games. 

Maker--Breaker games have attracted a lot of attention in the last few decades, and many exciting results about various games have been proven. The standard approach is to consider the games in which the board is the edge set of a given graph, and the winning sets are some graph theoretic structures, such as the spanning trees, Hamilton cycles, perfect matchings and others.  A well studied class of graphs are certainly the complete graphs on $n$ vertices, for sufficiently large integer $n$.  However,  it turns out that in many such games Maker wins quite easily when both players claim just one element per move. \cite{CE78} first observed this in their seminal paper, and introduced \emph{the biased games}, the games in which Maker claims $1$ element per move, while Breaker is allowed to claim $b\geq1$ elements per move.  That initiated various types of studies of games, where Breaker is given more power, such as already mentioned biased games (for reference see e.g.~\cite{GS09, HKSS09, K11}), or making the board sparser (see e.g.~\cite{BP12, HKSS11, SS05}), and other constraints on the way players choose the elements they claim (see e.g.~\cite{CT16, EFPK15,FM20}).

One can also consider Maker--Breaker games on graphs in which players claim vertices instead of edges in each turn. One such type is being recently introduced as the \emph{Maker--Breaker domination} game (in the combinatorial setup) by~\cite{DGPR18}, inspired by the domination game introduced by~\cite{BKR10}. A \emph{dominating set} $D$ of a given graph $G$ is a set of vertices of $G$, such that every vertex not in $D$ has a neighbour in $D$. In a Maker--Breaker domination game on a graph $G$, one player whom we shall call the \emph{Dominator}, claims an unclaimed vertex of $G$ in his turn, aiming at claiming the dominating set of a given graph $G$. The other player, named \emph{Staller}, claims another unclaimed vertex in her turn, with goal to prevent Dominator from achieving his goal. In other words, Staller's goal is to claim a closed neighbourhood of some vertex in $G$. The names Dominator and Staller are used to be consistent with the literature considering both standard domination game (for standard rules, see~\cite{BKR10} and recent results, see e.g.~\cite{BKR13, DKR15}) and Maker--Breaker domination games. 

Also, to keep pace with the current notation in this type of games, Staller will be Maker, and Dominator will play the role of Breaker, as introduced by~\cite{DGPR18}, and used later by~\cite{GHIK19}.

Maker--Breaker total domination game (MBTD game for short) is introduced as a natural counterpart of the Maker--Breaker domination game by~\cite{GHIK19}. Here, given the graph $G$, the Dominator's goal is to claim the \emph{total} dominating set of $G$, i.e.\ a set $D$ s.t.\ each vertex of $G$ has a neighbour in $D$. On the other hand, Staller wins if she claims an open neighbourhood of some vertex $v$ of $G$. 

It is easy to see that Dominator wins the MBTD game on $G=K_n$, for any given $n\geq 2$, no matter who starts the game. Also, for $G=K_1$, the MBTD game is a Staller's win, as the graph has a unique vertex. Therefore, it is interesting to look at other graphs and characterize them according to the winner of the game.\\ 
Following the notation from \cite{GHIK19} we say that a graph $G$ is
\begin{enumerate}
\item[•] $\cD$, if Dominator wins the game,
\item[•] $\cS$, if Staller wins the game,
\item[•] $\mathcal{N}$, if first player wins the game. 
\end{enumerate}

Cubic graphs have been studied for quite a while with the respect to the domination number (see e.g.~\cite{Z96}) and also in the domination game, already mentioned above, most recently by~\cite{BIK20}. Moreover, connected cubic graphs are very interesting to study in MBDT game setup, as already noticed by~\cite{GHIK19}. The authors showed that there are infinitely many connected cubic graphs which are Staller's win, provided that she plays first. The result is connected with the number of total dominating sets in which the vertex set of a given graph $G$ can be partitioned, introduced by~\cite{CDH80}, for which it was later shown by~\cite{DHH17} that infinitely many connected cubic graphs exist.  

Motivated by the question of~\cite{GHIK19}, we study MBTD game on connected cubic graphs ($3$-regular graphs) and for some classes of connected cubic graphs we are able to give exact characterization which ones are in $\cD$ and which are in $\cS$. From the paper of~\cite{GHIK19} we know that
\begin{enumerate}
\item[•] there are infinitely many examples of connected cubic graphs in which Staller wins the game as the first player, 
\item[•] no minimum degree condition is sufficient to guarantee that Dominator wins in the game in which Staller is the first player. 
\end{enumerate}

Some research on this topic has been conducted previously, although not under the name MBTD game. Namely, in his PhD thesis,~\cite{Kutz04} considered the Maker--Breaker games on almost-disjoint hypergraphs of rank three (edges with up to three vertices intersecting in at most one vertex), where  the players alternately claim vertices of a given hypergraph. In an almost disjoint hypergraph of rank three it can be decided in polynomial time whether Maker or Breaker wins, as shown by~\cite{Kutz04}.

Looking at the winning sets of Staller in our MBTD game, we would have a hypergraph whose all edges have three vertices (still a rank three hypergraph), but the key difference is that majority of the graphs that we consider are not almost-disjoint, as a lot of intersections exist among the edges. We also look at some graphs whose hypergraph would have almost-disjoint hyperedges (e.g. the union of at least four vertex-disjoint $K_{1,3}$). In this case, the analysis of~\cite{Kutz04} would require searching through all possible pairs of first moves, and then applying some reductions when Staller is the first player. However, to be able to classify these graphs, we provide in this paper a more applicable and explicit winning strategy for the players.

To determine which connected cubic graphs are Dominator's win and which are Staller's win, we use the following  classification of cubic graphs. In a cubic graph on $n\geq 6$ vertices each vertex has only three possibilities, as shown by~\cite{K84}:
\begin{enumerate}
\item[1.] it lies in two triangles (Figure \ref{f3a}),
\item[2.] it lies in one triangle (Figure \ref{f3b}),
\item[3.] it lies in zero triangles (Figure \ref{f3c}).
\end{enumerate}
\begin{figure}[h]
    \centering
    \subfigure[Vertex can lie in two triangles \label{f3a}]{\includegraphics[scale=0.8]{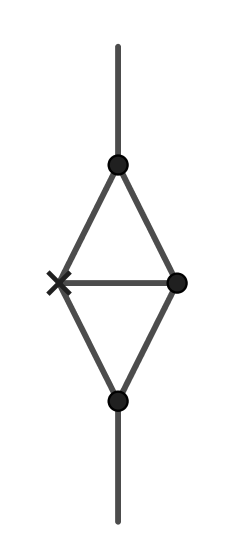}} 
    \hfill
    \subfigure[Vertex can lie in one triangles \label{f3b}]{\includegraphics[scale=0.8]{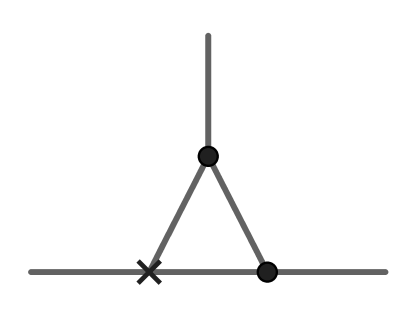}} 
     \hfill
    \subfigure[Vertex can lie in no triangles \label{f3c}]{\includegraphics[scale=0.8]{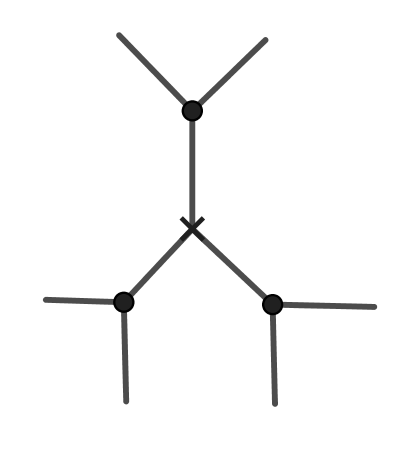}}
    \caption{Different possible locations for vertices in cubic graph of order $n\geq 6$.}
    \label{f3}
\end{figure}

So, cubic graphs can be classified according to the number of vertices of type 1 (being in two triangles), type 2 (being in one triangle) and type 3 (being in no triangle). Let $T_1, T_2$ and $T_3$ denote the number of vertices of type 1, type 2 and type 3, respectively. These three numbers are related by the following formulas of~\cite{K84}:
$$ T_1 = 2k_1 \quad T_2 = T_1 + 3k_2 \quad T_1 + T_2 + T_3 =n,$$
where $k_1$ and $k_2$ are nonnegative integers.  

If the cubic graph contains vertices of type 1, then this means that $G$ contains at least one \textit{diamond}, that is, the complete graph on four vertices minus one edge. 

Hereinafter, when we say a triangle we refer to an induced $K_3$ which is not part of a diamond. 
We say that a vertex $x \in V(G)$ is adjacent to some triangle $Y\subseteq G$, where $V(Y) = \{y_1, y_2, y_3\}$, if $xy_i \in E(G)$, for some $i\in \{1,2,3\}$. Also, we say that two triangles $X\subseteq G $ and $Y \subseteq G$, with the vertex sets $V(X) = \{x_1, x_2, x_3\}$ and $V(Y) = \{y_1, y_2, y_3\}$, are adjacent if $x_iy_j \in E(G)$ for some $i,j \in \{1,2,3\}$. 

Let $H$ be a fixed graph on $h$ vertices and let $n=x\cdot h$. We say that a graph $G$ on $n$ vertices contains an $H$-factor if it contains $x$ vertex disjoint copies of $H$.

Taking into consideration the possible types of cubic graphs, we prove the following theorems.
\begin{thm}\label{thm1}
Let $G$ be a cubic graph on $n\geq 6$ vertices containing a diamond-factor. Then, the graph $G$ is $\cD$. 
\end{thm}
\begin{thm}\label{thm2}
Let $G$ be a cubic graph on $n\geq 6$ vertices that contains a triangle-factor. If $n=6$, the graph $G$ is $\cD$. Otherwise, the graph $G$ is $\cS$.  
\end{thm}
\begin{thm}\label{thm3}
Let $G$ be a cubic graph on $n\geq 6$ vertices which contains vertex disjoint copies of triangles and diamonds. Then, there are only two types of such a graph on which Dominator wins, but only as the first player. In all other cases the graph $G$ is $\cS$. 
\end{thm}
\begin{defi}[\cite{W69}]
The generalized Petersen graph $GP(n,k)$, for given $n\geq 3$ and $1\leq k\leq n-1$, is a graph whose vertex set is $V(GP(n,k))= \{u_0, u_1, \dots, u_{n-1}, v_0, v_1, \dots, v_{n-1}\}$ and $E(GP(n,k))$ consists of all edges in form $$ u_iu_{i+1}, \,\,\,\,\,\, u_iv_i, \,\,\,\,\,\, v_iv_{i+k},$$
where $i$ is an integer. All subscripts are to be read modulo $n$. 
\end{defi}
Generalized Petersen graphs drew lots of attention since their definition. \cite{GHIK19} showed that the prism $P_2 \square C_n$, for $n\geq 3$ is $\cD$. 
This graph is isomorphic to the generalized Petersen graph $GP(n,1)$. For the graph $GP(5,2)$ it is proven by~\cite{GHIK19} that it is $\mathcal{S}$. The following theorem gives the characterization of $GP(n,2)$ for every $n\geq 6$. 
\begin{thm}\label{thm4}
The generalized Petersen graph $GP(n,2)$, where $n\geq 6$, is $\mathcal{S}$. 
\end{thm}
Finally, we are interested in MBTD games on cubic bipartite graphs and the union of bipartite graphs and prove the following. 
\begin{thm}\label{thm5}
A cubic bipartite graph is $\mathcal{D}$.
\end{thm}
A \textit{claw} is the complete bipartite graph $K_{1,3}$.
\begin{thm}\label{thm6}
Let $G$ be a connected cubic graph on $n\geq 6$ vertices containing the claw-factor, and let $k\geq 2$ denote the number of vertex disjoint claws in $G$. For $k = 2$, $G$ is $\mathcal{D}$. For $k\geq 3$, the graph $G$ is $\mathcal{S}$. 
\end{thm}

\subsection{Preliminaries}
For given graph $G$ by $V(G)$ and $E(G)$ we denote its vertex set and edge set, respectively. The order of graph $G$ is denoted by $v(G) = |V(G)|$, and the size of the graph by $e(G) = |E(G)|$. 

Assume that the MBTD game is in progress.
We denote by  $d_1, d_2,...$ the sequence of vertices chosen by
Dominator and by $s_1, s_2,...$ the sequence of vertices chosen by Staller. 
At any given moment during this game, we denote the set of vertices claimed by Dominator by $\mathfrak{D}$ and the set of vertices claimed by Staller by $\mathfrak{S}$. As in the paper of~\cite{GIK19}, we say that the game is the $D$-game if Dominator is the first to play, i.e. one \textit{round} consists of a move by Dominator followed by a move of Staller.
In the $S$-game, one round consists of a move by Staller followed by a move of Dominator. 
The vertices in $V(G) \setminus(\mathfrak{D}\cup \mathfrak{S})$ are called \emph{free}. 
We say that the vertex $v$ is \textit{isolated} by Staller if  all neighbours of $v$ are claimed by Staller. 
The open neighbourhood of a vertex $v$, denoted by $N_G(v)$, is the set of vertices adjacent to $v$ in $G$. Graph $G$ is $r$-regular if every vertex $v\in V(G)$ has degree $r$. A $3$-regular graph is called \textit{cubic graph}.   
\begin{defi}
The Cartesian product $G \square H$ of graphs $G$ and $H$ is the graph with
vertex set $V (G \square H)$ = $V(G)\times V(H)$ in which $(u,v)$ is adjacent to $(u',v')$ if either $u = u'$ and $vv' \in E(H)$, or $v=v'$ and $uu'\in E(G)$. \\ The circular ladder graph (or prism graph) $CL_n$ is the Cartesian product of a cycle of length $n \geq 3$ and an edge, that is, $CL_n = C_n \square P_2$. 
\end{defi}
\noindent A $n$-prism graph is isomorphic to the generalized Petersen graph $GP(n,1)$. \\ \\
We point out some basic properties of the MBTD games given by~\cite{GHIK19}. 
\begin{proposition}\label{cor2}(\cite{GHIK19}, Corollary 2.2(ii))
Let $G$ be a graph. 
Let $V_1, ..., V_k$ a partition of $V(G)$ such that $V_i$,
$i \in [k] := \{1,...,k\}$, induces a graph on which Dominator wins the MBTD game, then Dominator wins MBTD game on $G$. 
\end{proposition}
\begin{proposition}\label{prop}(\cite{GHIK19}, Proposition 2.4)
Dominator wins in MBTD game on cycle $C_4$. 
\end{proposition}
\subsubsection{Traps}
Consider the MBTD game on a graph $G$. Let $v\in V(G)$ and let $u_1, u_2, u_3 \in N_G(v)$. Let $u_2$ and $u_3$ be free vertices. 
Let $u_1 \in \mathfrak{S}$ and suppose that it is Staller's turn to make her move. If Staller claims $u_2$ (or $u_3$), she creates a \textit{trap} for Dominator, that is, Staller forces Dominator to claim $u_3$ (or $u_2$) as otherwise she isolates $v$. 
\paragraph{Double trap.}
We say that Staller creates a \textit{double trap} $u-v$ in the MBTD game on $G$, where $u,v \in V(G)$ are free vertices, if after Staller's move Dominator is forced to claim both vertices $u$ and $v$. Since Dominator can not claim two vertices in one move, in her next move Staller will claim either $u$ or $v$ and isolate either a neighbour of $u$ or a neighbour of $v$. 
If Staller creates a double trap, Dominator loses the game. 

For an illustration, consider subgraph $H_1$ on Figure~\ref{trapsa}. Suppose that $v_2, v_3 \in \mathfrak{S}$ and $v_0 \in \mathfrak{D}$. If Staller claims a vertex $v_1$, she creates a double trap $u-v$. In her next move, Dominator is forced to claim vertex $u$ to prevent $v_2$ from being isolated by Staller in her next move, but also Dominator is forced to claim $v$ to prevent $v_3$ from being isolated by Staller in her next move. 
\paragraph{Diamond trap.} Suppose that the MBTD game on the connected cubic graph $G$ is in progress. 
Let $Z \subseteq G$ be a diamond with the vertex set $V(Z) = \{z_1, z_2, z_3, z_4\}$ and the edge set $E(Z)= \{z_1z_2, z_2z_3, z_3z_4, z_4z_1,  z_2z_4\}$ and suppose that all vertices from $V(Z)$ are free. 
By claiming $z_1$ (or $z_3$) Staller creates a \textit{diamond trap} on $Z$. That is, she forces Dominator to claim a vertex from $V(Z) \setminus \{z_1\}$ (or $V(Z) \setminus \{z_3\}$), as otherwise, in her next move Staller will claim $z_3$ (or $z_1$) and create a double trap $z_2-z_4$.
\sloppy  
\paragraph{Vertex-diamond trap.}
Suppose that the MBTD game on the connected cubic graph $G$ is in progress. Consider subgraph $G' \subseteq G$ with the vertex set $V(G') = \{v, y_1, y_2, z_1, z_2, z_3, z_4\}$, where the vertices $z_1, z_2, z_3$ and $z_4$ form a diamond $Z$ with the edge set $E(Z)=\{z_1z_2, z_2z_3, z_3z_4, z_4z_1, z_2z_4\}$. Let $E(G') = E(Z) \cup \{vy_1, vy_2, vz_1\}$. Let $y_2, z_1, z_2, z_3, z_4$ be free vertices. Suppose that $y_1 \in \mathfrak{S}$ and it is Staller's turn to make her move. 
If Staller claims $z_1$ she creates a \textit{vertex-diamond trap} $y_2-Z$. That is,  
she forces Dominator to claim $y_2$, as otherwise Staller can isolate $v$ in her next move. Also, Staller has created a diamond trap on $Z$ which forces Dominator to claim $V(Z) \setminus \{z_1\}$. In any case Dominator will lose the game. 

For an illustration, consider subgraph $H_2$ on Figure~\ref{trapsb}. Suppose that $y_1 \in \mathfrak{S}$. Once Staller claims vertex $z_1$, she forces Dominator to claim one of the vertices from $\{z_2, z_3, z_4\}$ as she created a diamond trap on $Z$, but also she forces Dominator to claim vertex $y_2$ to prevent $v$ from being isolated in Staller's next move. 
\begin{figure}[htbp]
  \begin{center}
    \subfigure[Subgraph $H_1$ \label{trapsa}]{\includegraphics[scale=0.2]{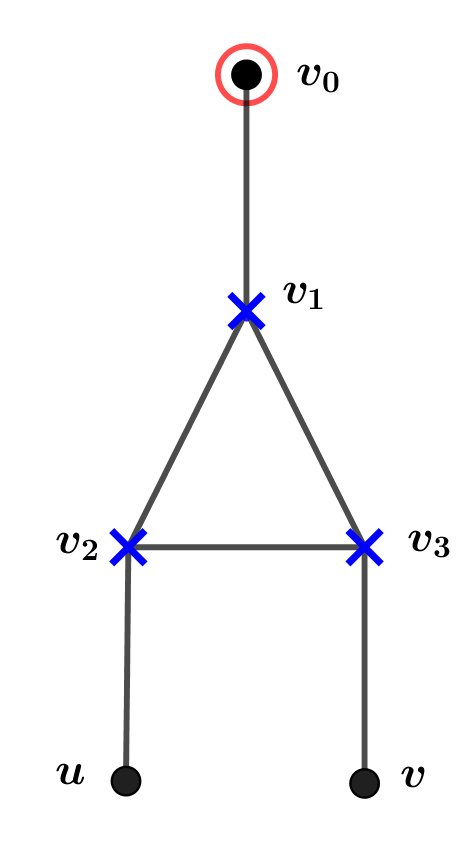}}
    \hfil
    \subfigure[Subgraph $H_2$ \label{trapsb}]{\includegraphics[scale=0.2]{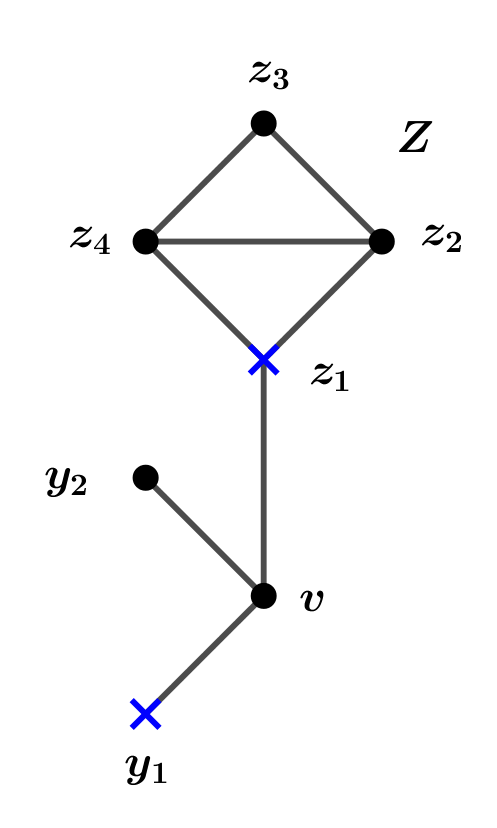}}
    \caption{Traps illustration. \\ 
   Staller's moves are denoted by blue crosses and Dominator's move $v_0$ is denoted by red circle.}
    \label{traps}
  \end{center}
\end{figure}

\subsubsection{Pairing strategy for Dominator.}

Pairing strategy is one of the most natural strategies for Breaker, as defined by~\cite{HKSSBook14}. 
In order to win in MBTD game played on certain graphs, Dominator will use the pairing strategy. This means that the subset of the board of the game can be partitioned into pairs such that every winning set (i.e.\ open neighbourhood of a vertex in the graph) contains one of the pairs. When Staller claims an element from some pair, Dominator will respond by claiming the other element from that pair. 

Pairing strategy has been already studied for Maker--Breaker domination games by~\cite{DGPR18}, where it is proven that if a graph $G=(V, E)$ admits a pairing dominating set, i.e.\ a set $\{(u_1, v_1), (u_2, v_2), \dots, (u_k, v_k)\}\subseteq V$, where all vertices are distinct and $V=\cup_{i=1}^{k}N[u_i]\cap N[v_i]$, then the graph is $\mathcal{D}$ (Proposition 9 by~\cite{DGPR18}). 

\section{Graphs from $\mathcal{D}$ and $\mathcal{S}$} 
\begin{proof}[of Theorem \ref{thm1}] 
The vertex set $V(G)$ can be partitioned into $4$-sets, each containing a $C_4$. So, by Proposition \ref{prop} and Proposition \ref{cor2} Dominator wins on $C_4$, and therefore on diamond. 
\end{proof}

\begin{defi}
\label{def1}
Suppose that the MBTD game on the connected cubic graph $G$ on $n\geq 6$ vertices is in progress.
\begin{enumerate}
\item[1.] By $G_1$ denote an induced subgraph of $G$ with the vertex set $V(G_1) = \{u_0,u_1,u_2,u_3, v_0,v_1,v_2,v_3 \}$, where the vertices $u_1,u_2,u_3$ form a triangle $U$ and the vertices $v_1,v_2,v_3$ form a triangle $V$.  Let $E(G_1) = E(U) \cup E(V) \cup  \{u_0u_1,v_0v_1, u_2v_2, u_3v_3\}$.
The subgraph is illustrated in Figure \ref{subgrapha}. 
\item[2.] By $G_2$ denote a subgraph of $G$ with the vertex set $\{x_1,x_2,x_3, u_1,u_2,u_3, v_1,v_2,v_3, z_1,z_2,z_3,v\}$, where the vertices 
$x_1,x_2,x_3$ from a triangle $X$, $u_1,u_2,u_3$ form a triangle $U$, $v_1,v_2,v_3$ form a triangle $V$, and $z_1,z_2,z_3$ form a triangle $Z$. 
Let $E(G_2) = E(X) \cup E(U) \cup E(V) \cup E(Z) \cup \{v_3v,u_1x_1, u_2v_2, u_3z_3\}$. 
It could be a case that the neighbour of vertex $v_3$ denoted by $v$ is one of the vertices from the set $\{x_2, x_3, z_1, z_2\}$. 
The subgraph is illustrated in Figure \ref{subgraphb}. 
\item[3.] By $G_3$ denote a subgraph of $G$ which contains a triangle $U$ with the vertex set $\{u_1,u_2,u_3\}$, and a diamond $Z$ with the vertex set $\{z_1,z_2,z_3,z_4\}$ and the edge set $E(Z)= \{z_1z_2,z_2z_3,z_3z_4, z_4z_1, z_2z_4\}$. Let $E(G_3) = E(U) \cup E(Z) \cup \{u_2z_1\}$. 
The subgraph is illustrated in Figure \ref{subgraphc}. 
\item[4.] By $G_4$ denote a subgraph of $G$ which contains a triangle $U$ with the vertex set $\{u_1,u_2,u_3\}$ and two diamonds, a diamond $Y$ with the vertex set $\{y_1,y_2,y_3,y_4\}$ and the edge set $E(Y)=\{y_1y_2,y_2y_3,y_3y_4,y_2y_4\}$, and a diamond $Z$ with the vertex set $\{z_1,z_2,z_3,z_4\}$ and the edge set $E(Z)=\{z_1z_2,z_2z_3,z_3z_4,z_4z_1,z_2z_4\}$. Let $E(G_4) = E(U) \cup E(Y) \cup E(Z) \cup \{u_2y_1, u_3z_1\}$. 
The subgraph is illustrated in Figure \ref{subgraphd}.
\end{enumerate}
 \begin{figure}[!h]
    \centering
    \subfigure[$G_1$ \label{subgrapha}]{\includegraphics[scale=0.7]{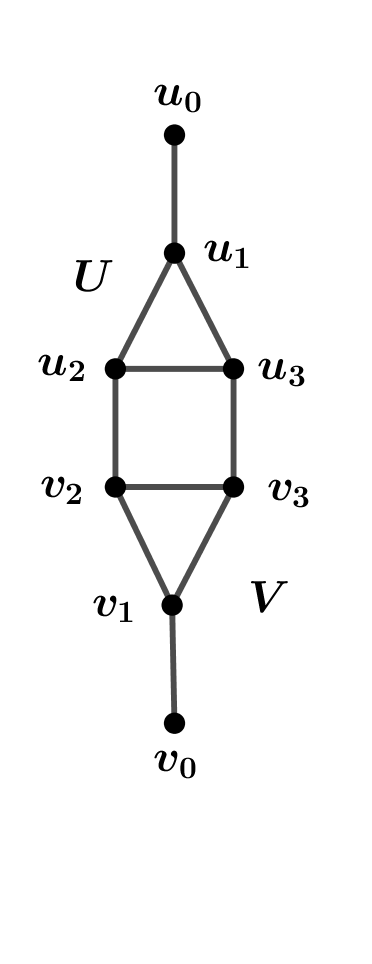}} 
    \hfill
    \subfigure[$G_2$ \label{subgraphb}]{\includegraphics[scale=0.22]{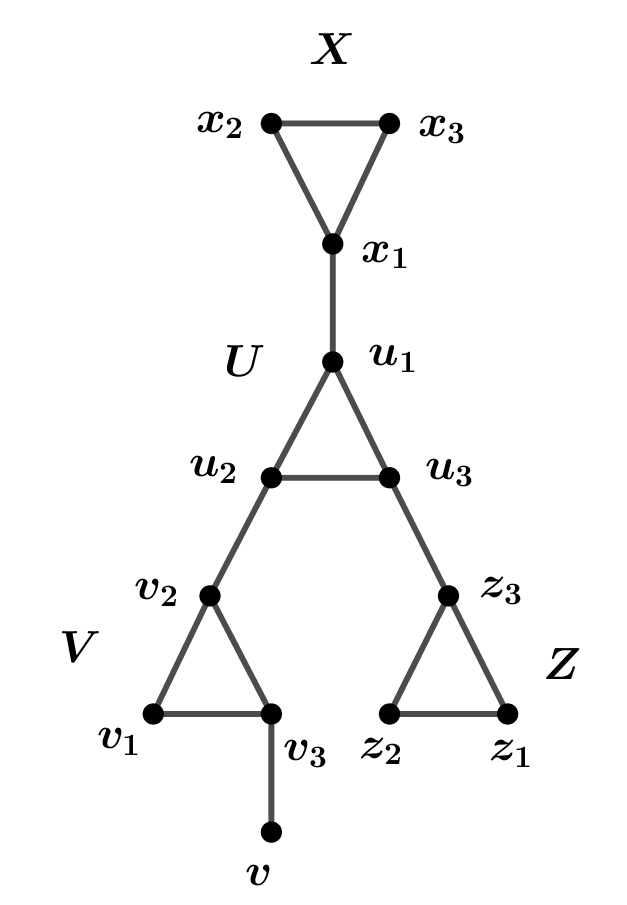}} 
    \hfill
    \subfigure[$G_3$ \label{subgraphc}]{\includegraphics[scale=0.7]{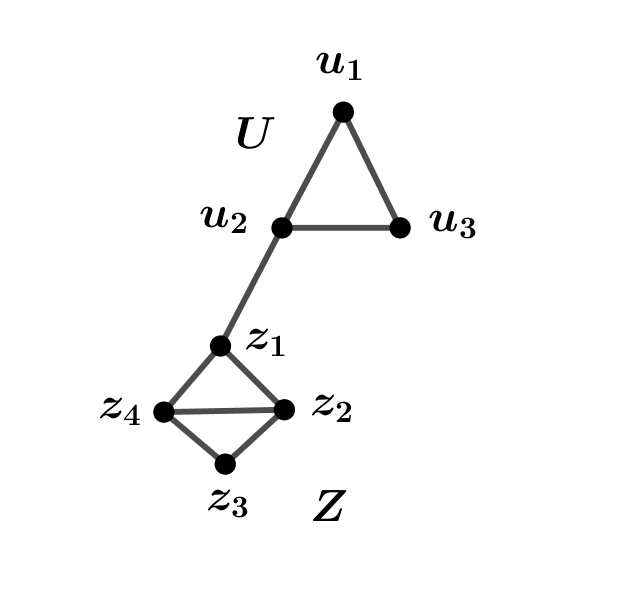}}
    \hfill
    \subfigure[$G_4$ \label{subgraphd}]{\includegraphics[scale=0.7]{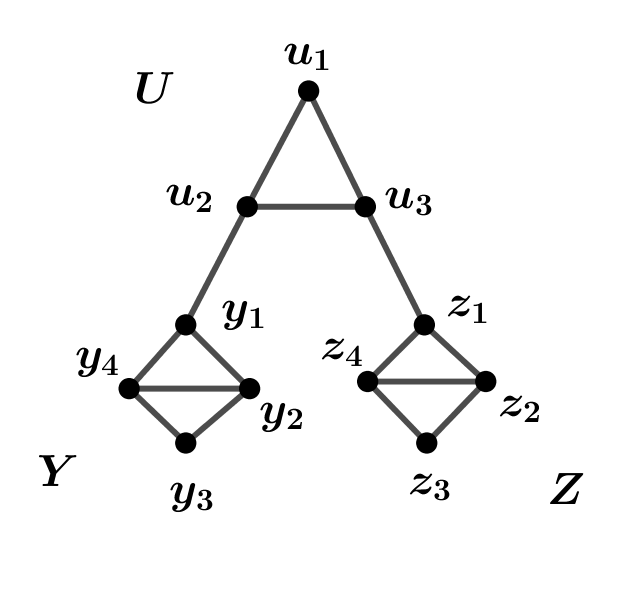}}
    \caption{Subgraphs illustration from Definition \ref{def1}.}
    \label{subgraph}
\end{figure}
\end{defi}
To prove Theorem \ref{thm2} and Theorem \ref{thm3} we will need the following lemmas. 
\begin{lem}\label{lema1}
Consider the $D$-game on $G_1$. 
If $d_1 = u_0$, then $G_1$ is $\mathcal{S}$. Also, Staller wins $S$-game on $G_1$. 
\end{lem}
\begin{proof}
We have $d_1 = u_0$. Then, $s_1 = v_2$. Consider the following cases: 
\begin{enumerate}[{Case }1.]
\item $d_2 \in \{u_1, u_3, v_1\}$. Then, $s_2 = v_3$ which forces $d_3 = v_0$. 
In her third move, if $d_2 = u_1$, Staller claims $v_1$ and creates a double trap $u_2-u_3$,
if $d_2 = u_3$, Staller claims $u_2$ and creates a double trap $u_1-v_1$, and if $d_2 = v_1$, Staller claims $u_1$ Staller and creates a double trap $u_3-u_2$. 
\item $d_2 \in \{u_2, v_3, v_0\}$. Then, by playing $s_2 = u_3$ Staller creates a double trap $u_1-v_1$. In her third move Staller isolates $u_2$ or $v_3$.
\end{enumerate}
In the $S$-game, Staller can pretend that she is the second player and $d_1=u_0$ and win the game.
\end{proof}

\begin{lem}\label{lema2}
Consider the MBTD game on $G_2$. Let $u_1\in \mathfrak{S}$ and suppose that at least the vertices $u_2, u_3, v_1, v_2, v_3, v, z_3$ are free. 
Suppose that it is Staller's turn to make her move. Then, Staller wins. 
\end{lem}
\begin{proof}
Staller plays in the following way: $s_1 = u_2$ which forces $d_1 = z_3$ and $s_2 = v_2$ which forces $d_2 = u_3$. By playing $s_3 = v_1$ Staller creates a double trap $v_3-v$. In her next move Staller isolates either $v_2$ or $v_3$ by claiming $v_3$ or $v$. 
\end{proof}

\begin{lem}\label{lema3}
Consider the MBTD game on $G_3$. Let $u_1\in \mathfrak{S}$ and suppose that at least the vertices $u_3, z_1, z_2, z_3, z_4$ are free.  
Suppose that it is Staller's turn to make her move. Then, Staller wins.  
\end{lem}
\begin{proof}
Staller plays $s_1 = z_1$ and creates a vertex-diamond trap $u_3-Z$. Dominator can not win. 
\end{proof}

\begin{lem}\label{lema4}
Staller wins the $S$-game on $G_4$.
\end{lem}
\begin{proof}
Consider the $S$-game on $G_4$. 
Staller plays in the following way: $s_1 = y_1$ which forces $d_1 \in \{y_2,y_3,y_4\}$ (a diamond trap on $Y$), 
as otherwise Staller will claim $y_3$ in her second move and then in her third move she can isolate $y_2$ or $y_4$. 
Next, $s_2 = u_1$ which forces $d_2 = u_3$. By playing
$s_3 = z_1$ Staller creates a vertex-diamond trap $u_2-Z$.  
\end{proof}

\begin{lem}\label{lema5}
Consider the MBTD game on the connected cubic graph $\eta $ illustrated in Figure \ref{f4}. In the $D$-game, if $d_1 \in \{h_1, h_3\}$, Dominator wins. Otherwise, Staller wins as the second player. In the $S$-game on $\eta $ Staller wins. 
\begin{figure}[!h]
\centering
 \includegraphics[scale=0.7]{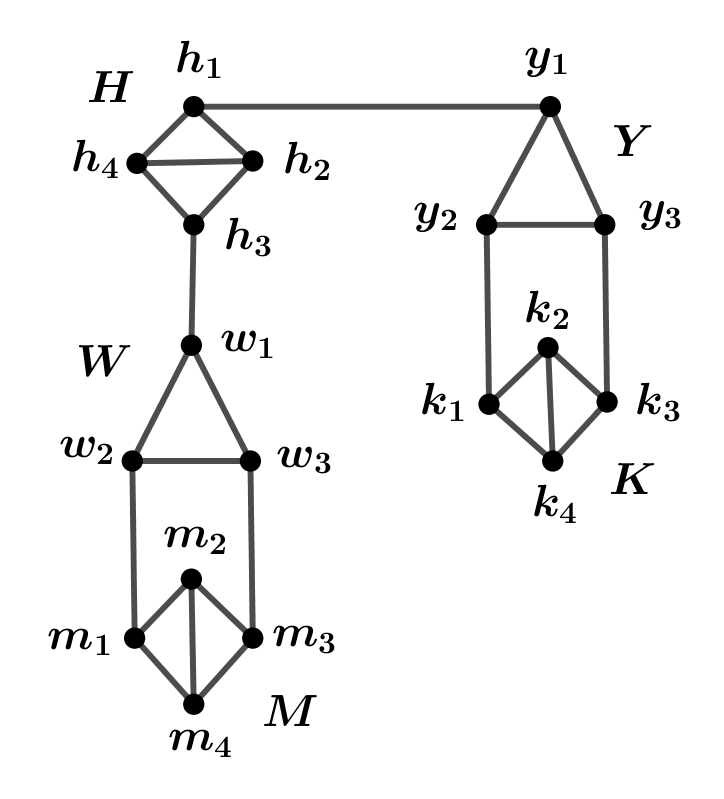}
 \caption{The graph $\eta $.}
 \label{f4}
\end{figure}
\end{lem}
\begin{proof}
Consider the $D$-game and let $d_1 \in \{h_1, h_3\}$. W.l.o.g.\ let $d_1 = h_1$.  
To dominate vertices from $V(Y) \cup V(K)$ Dominator plays in the following way:
\begin{enumerate}[{Case }1.]
\item If Staller's first move on $V(Y) \cup V(K)$ is $y_1$, Dominator responds with $k_1$. To cover the remaining vertices from $V(Y) \cup V(K)$, Dominator will use the pairing strategy on pairs $(y_2, k_3)$ and  $(k_2, k_4)$. If Staller claims $y_3$, Dominator will claim an arbitrary free vertex among $\{y_2,k_2,k_3,k_4\}$.
\item If Staller's first move on $V(Y) \cup V(K)$ is a vertex from $\{y_2, y_3\}$, Dominator responds with $y_1$. To dominate the remaining vertices from $V(Y) \cup V(K)$, Dominator will use the pairing strategy on pairs $(k_1, k_3)$ and  $(k_2, k_4)$.
\item If Staller's first move on $V(Y) \cup V(K)$ is $k_1$ (or $k_3$), Dominator responds with $k_3$ (or $k_1$). To cover the remaining vertices from $V(Y) \cup V(K)$, Dominator will use the pairing strategy, where the pairs are $(y_1, y_3)$ (or $(y_1, y_2)$) and  $(k_2, k_4)$. If Staller claims $y_2$ (or $y_3$), Dominator will claim an arbitrary free vertex from one of the pairs. 
\end{enumerate}
To cover vertices from $V(H) \cup V(W) \cup V(M)$ Dominator plays in the following way. 
First, he makes a pairing $(h_2, h_4)$ and $(m_2, m_4)$, and applies a pairing strategy there. So, suppose, $h_2, m_2 \in \mathfrak{D}$. It is enough to consider the following cases. 
\begin{enumerate}[{Case }1.]
\item  If Staller's first move on $V(H) \cup V(W) \cup V(M)$ is $h_3$, Dominator responds with $w_1$. In order to cover vertices from $V(W) \cup V(M)$ Dominator will use the pairing strategy on pairs $(w_2, w_3)$ and $(m_1, m_3)$. 
\item If Staller's first move on $V(H) \cup V(W) \cup V(M)$ is $w_1$ (or $w_2$), Dominator responds with $m_1$ (or $m_3$). In order to cover vertices from $V(H) \cup V(W) \cup V(M)$ Dominator will use the pairing strategy on pairs $(w_2, m_3)$ (or $(w_1, m_1)$)  and $(w_3, h_3)$. 
\item If Staller's first move on $V(H) \cup V(W) \cup V(M)$ is $w_3$, Dominator responds with $m_1$. In order to cover vertices from $V(H) \cup V(W) \cup V(M)$ Dominator will use the pairing strategy on pairs $(w_2, h_3)$ and $(w_1, m_3)$.
\item If Staller's first move on $V(H) \cup V(W) \cup V(M)$ is $m_1$ (or $m_3$), Dominator responds with $m_3$ (or $m_1$). In order to cover vertices from $V(H) \cup V(W) \cup V(M)$ Dominator will use the pairing strategy on pairs $(w_1, w_3)$ and $(w_2, h_3)$
(or $(w_1, w_2)$ and $(w_3, h_3)$).
\end{enumerate}
Next, suppose that $d_1 \in \{h_2, h_4\}$. W.l.o.g.\ let $d_1 = h_2$.  
Staller plays in the following way: $s_1 = k_1$ which forces $d_2 \in V(K)\setminus \{k_1\}$ (a diamond trap), $s_2 = y_3$ which forces $d_3 = y_1$, $s_3 = h_1$ which forces $d_4 = y_2$ and $s_4 = h_3$ which forces $d_5 = h_4$. Next, $s_5 = w_1$. Afterwards,
\begin{enumerate}
\item[-] if $d_6 = w_2$ (or $d_6 = w_3$), then $s_6 = m_1$ (or $s_6 = m_3$) and Staller creates a vertex-diamond trap $w_3-M$ (or $w_2-M$). 
\item[-] if $d_6 \in \{m_1, m_2, m_4\}$ (or $d_6 = m_3$), then $s_6 = w_2$ (or $s_6 = w_3$) and Staller creates a double trap $m_3-w_3$ (or $m_1-w_2$). In her next move Staller isolates either $w_3$ or $w_1$ (or, $w_2$ or $w_1$).  
\end{enumerate}
Next, suppose that $d_1 \notin V(H)$. W.l.o.g.\ let $d_1 \in V(Y) \cup V(K)$. Then, Staller plays on the subgraph $G_4$ with the vertex set $V(W) \cup V(H) \cup V(M)$. By Lemma \ref{lema4}, Staller wins.
In the the $S$-game, Staller  uses the same strategy. 
\end{proof}

\begin{lem}\label{lema6}
Consider the MBTD game on the connected cubic graph $\omega $ illustrated in Figure \ref{f5}, where the chain of diamonds adjacent to $A$ can consist of one or more diamonds. In the $D$-game, if $d_1 = a_1$ Dominator wins, if $d_1 \in V(D_1)$ Staller wins.  In the $S$-game on $\omega $ Staller wins. 
\begin{figure}[!h]
\centering
 \includegraphics[scale=0.7]{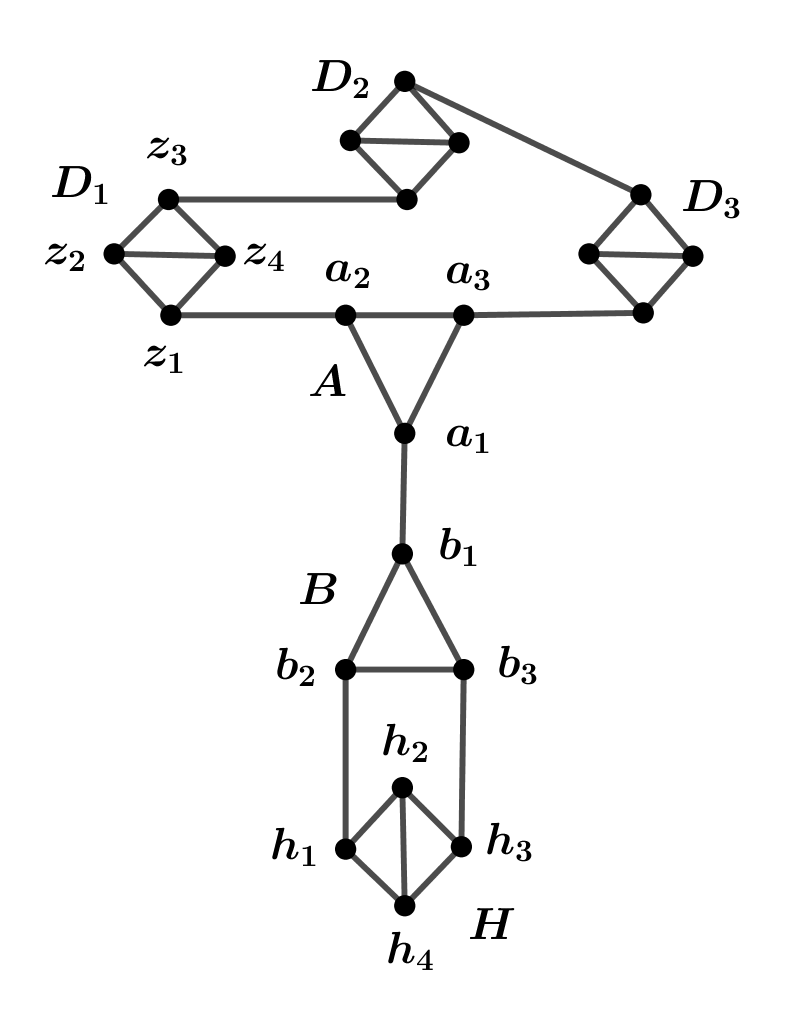}
 \caption{An example of a graph $\omega$, where the chain of diamonds consists of three diamonds.}
 \label{f5}
\end{figure}
\end{lem}
\begin{proof}
We first look at the $D$-game on $\omega $. In his first move Dominator claimed $a_1$.
When Staller plays on  a diamond which is different from $H$, Dominator will apply a pairing strategy on that diamond, where one pair consists of two opposite vertices of that diamond and remaining two vertices in diamond form the other pair. 
Dominator also makes a pairing $(a_2, a_3)$ and $(h_2, h_4)$, and apply the pairing strategy when Staller plays there. Let $a_2, h_2 \in \mathfrak{D}$. 
\begin{enumerate}[-]
    \item If Staller's first move on $V(B)\cup V(H)$ is $b_1$, Dominator responds with $h_1$. If Staller in her next move claims a vertex from $\{b_2, b_3, h_3\}$, Dominator claims the free vertex from $\{b_2, h_3\}$ and in this way he covers all vertices from the graph.
\item If Staller's first move on $V(B)\cup V(H)$ is $b_2$ (or $b_3$), Dominator responds with $h_1$ (or $h_3$). If Staller in her next move claims a vertex from $\{b_1, h_3, b_3\}$ (or $\{b_1, h_1, b_2\}$), Dominator claims the free vertex from $\{b_1, h_3\}$ (or $\{b_1, h_1\}$) and he covers all vertices. 
\item If Staller's first move on $V(B)\cup V(H)$ is $h_1$ (or $h_3$), Dominator responds with $h_3$ (or $h_1$). In his next move Dominator claims a free vertex from $\{b_1, b_3\}$ (or $\{b_1, b_2\}$). 
\end{enumerate}
Let $d_1\in V(D_1)$. Then, Staller, in her first move plays 
$s_1=b_1$. If $d_2\in V(A)\cup V(B) \cup V(D_i)$, then Staller can create $b_3-H$ trap or $b_2-H$ trap by claiming $h_1$ if $d_2=b_2$, or $h_3$ if $d_2=b_3$. So, Dominator needs to play his second move on $H$. 
If $d_2\in \{h_1, h_2, h_4\}$ (or $d_2=h_3$), then $s_2=b_2$ (or $s_2=b_3$) forcing $d_3=h_3$ (or $d_3=h_1$). Next, $s_3=a_1$ forcing $d_4=b_3$ (or $d_4=b_2$). \\
If chain of diamonds consist only of one diamond, say $D_1$, then if $d_1 \in \{z_1, z_2, z_4\}$ Staller plays $s_4=a_2$ creating $z_3-a_3$ trap, and if $d_1=z_3$ Staller plays $s_4=a_3$ creating $z_1-a_2$. 
Otherwise, if chain has more than one diamond, then Staller claims a vertex from a diamond incident with $a_3$ and creates vertex-diamond trap.  \\ \\
Consider the $S$-game. Staller plays in the following way: $s_1 = h_1$ which forces $d_1 \in V(H) \setminus \{h_1\}$ (a diamond trap), $s_2 = b_1$ which forces $d_2 = b_3$, $s_3 = z_1$ which forces $d_3 \in V(D_1) \setminus \{z_1\}$ (a diamond trap) and $s_4 = a_3$. Staller creates a double trap $a_1-a_2$.
\end{proof}

In the following we look at the graph on $n$ vertices that consists on vertex-disjoint triangles and prove that Staller wins even as the second player if $n>6$. 

\begin{proof}[of Theorem \ref{thm2}]
First, let $n=6$. Consider the $S$-game. 
Let $U$ be a triangle with the vertex set $\{u_1,u _2, u_3\}$ and let $V$ be a triangle with the vertex set $\{v_1,v _2, v_3\}$. Let $u_iv_i \in E(G)$ for every $i\in \{1,2,3\}$. W.l.o.g.\, suppose that Staller in her first move chooses a vertex $u_1$. Then Dominator will choose a vertex from the opposite triangle $V$ which is not adjacent to $u_1$, say a vertex $v_2$. In her second move Staller needs to claim $u_2$, as otherwise Dominator will win after his second move. If $s_2 = u_2$, then $d_2 = v_3$. In his third move Dominator will claim a free vertex from $\{v_1, u_3 \}$ and win. One of these two vertices must be free after Staller's third move. \\ \\ 
Let $n>6$. Since the graph is cubic, the number of vertices needs to be even, so we have an even number of triangles. 
Consider the $D$-game on graph $G$. Suppose that in his first move Dominator claims some vertex $a_1$ which belongs to a triangle $A$ with the vertex set $V(A) = \{a_1, a_2, a_3\}$. 
\paragraph{Case 1} Let $a_1b_1 \in E(G)$, where $b_1$ is a vertex of some triangle $B$ with the vertex set $V(B) = \{b_1, b_2,b_3\}$ and there is only one edge between $A$ and $B$.
We consider the following subcases. 
\paragraph{1.i.} Triangle $B$ is adjacent to one more triangle, say $Y$ with the vertex set $V(Y)= \{y_1, y_2, y_3\}$. So, there are two edges between $B$ and $Y$, say $b_2y_2$ and $b_3y_3$.
Let $y_1' \in N_G(y_1)$ for some $y_1' \in V(G)\setminus \{y_2, y_3\}$. Consider a subgraph induced by $\{a_1, b_1, b_2, b_3, y_1, y_2, y_3, y_1'\}$, where $a_1$ is claimed by Dominator and now it is Staller's turn to make her move. By Lemma \ref{lema1} it follows that Staller wins.  
\paragraph{1.ii.} Triangle $B$ is adjacent to two more triangles, say $Y$ with the vertex set $V(Y)= \{y_1, y_2, y_3\}$, and $W$ with the vertex set $V(W)= \{w_1, w_2, w_3\}$. Let $b_2w_2, b_3y_3 \in E(G)$. 
Let $w_1' \in N_G(w_1)$ for some $w_1' \in V(G)\setminus \{w_2, w_3\}$ and let $w_3' \in N_G(w_3)$ for some $w_3' \in V(G)\setminus \{w_1, w_2\}$. \\
In her first move Staller plays $s_1 = b_1$. The rest of the Staller's strategy depends on Dominator's second move. So, we analyse the following cases.
\begin{enumerate}
\item[1.ii.1.] $d_2 \in V(A)$. \\
Then, $s_2 = b_2$ which forces $d_3 = y_3$, $s_3 = w_2$ which forces $d_4 = b_3$. 
If $w_1' = a_i$ (or $w_3' = a_i$), for some $i\in \{2,3\}$, and $a_i$ is claimed by Dominator in his second move, then $s_4 = w_1$ (or $w_3$). In this way Staller creates a double trap $w_3-w_3'$ (or $w_1-w_1'$). In her next move Staller isolates either $w_2$ or $w_3$ (or, either $w_2$ or $w_1$).
Otherwise, Staller can claim any of the vertices $w_1$, $w_3$ in her fourth move and then play in the same way as above, i.e. she creates a double trap and wins in the following move. 
\item[1.ii.2.] $d_2 \in V(B)$. \\
W.l.o.g.\, let $d_2 = b_3$. 
\item[1.ii.2.1.] Triangle $Y$ is adjacent to two more triangles, say $K$ with the vertex set $V(K) = \{k_1,k_2,k_3\}$ and $H = \{h_1,h_2,h_3\}$. Let $y_2k_2 \in E(G)$ and $y_1h_1 \in E(G)$.  
Then, $s_2 = y_3$ which forces $d_3 = b_2$, $s_3 = y_1$ which forces $d_4 = k_2$ and $s_4 = h_1$ which forces $d_5 = y_2$. Next, $s_5 = h_2$ and Staller creates a double trap $h_3-h_3'$, where $h_3' \in N_G(h_3)$ for some $h_3' \in V(G) \setminus \{h_1, h_2\}$. In her next move Staller isolates either $h_1$ or $h_3$. \\
The statement holds also if $h_3' \in (V(K)\setminus \{k_2\}) \cup (V(W)\setminus \{w_2\}) \cup (V(A)\setminus \{a_1\}) $. \\
It could be the case that one of these triangles $K$, $H$ is the triangle $W$. The statement also holds in this case. \\
If $K=A$, the statement also holds. If $H=A$, the proof is very similar, but simpler, as the Staller wins in her fifth move. 
\item[1.ii.2.2.] Triangle $Y$ is adjacent to one more triangle, say $K\neq A$ with the vertex set $V(K) = \{k_1,k_2,k_3\}$. Let $y_1k_1, y_2k_2 \in E(G)$. Assume that $k_3'\in N_G(k_3)$ for some $k_3' \in V(G) \setminus \{k_1, k_2\}$. 
Since the graph induced by $\{b_3, y_1, y_2, y_3, k_1, k_2, k_3, k_3'\}$ is a variant of graph $G_1$, where $b_3 \in \mathfrak{D}$. According to Lemma \ref{lema1}, Staller wins. \\ 
If $K=W$, the statement also holds. 
\item[1.ii.2.3.] Triangle $Y$ is adjacent to triangle $A$ and there are two edges between them, say $y_2a_2, y_1a_3$. \\
Then, Staller plays  $s_2 = y_3$ which forces $d_3 = b_2$. By playing $s_3 = a_2$ Staller creates a double trap $a_3-y_1$. In her next move Staller isolates either $a_1$ or $y_2$. 
\item[1.ii.3.] $d_2 \in V(Y) \cup V(W)$. W.l.o.g.\ let $d_2 \in V(Y)$. \\
Let $w_1' \in N_G(w_1)$ and let $w_3' \in N_G(w_3)$. Consider the following subcases. 
\item[1.ii.3.1.] $d_2 \neq y_3$. W.l.o.g.\, let $d_2 = y_1$. \\
Then, $s_2 = w_2$ which forces $d_3 = b_3$ and $s_3 = b_2$ which forces $d_4 = y_3$. It is enough to consider the case if one of the vertices $w_1'$ and $w_3'$ is the vertex $y_1$. Let, for example, $w_1' = y_1$. Then, $s_4 = w_1$ and Staller creates a double trap $w_3-w_3'$. In her next move Staller isolates either $w_2$ or $w_3$. \\
Staller plays in the same way if $y_1 \notin \{w_1', w_3'\}$.
\item[1.ii.3.2.] $d_2 = y_3$. \\
Then, $s_2 = w_2$ which forces $d_3 = b_3$, $s_3 = w_1$ which forces $d_4 = w_3'$ and $s_4 = w_3$. Staller creates a double trap $b_2-w_1'$. In her next move Staller isolates either $w_2$ or $w_1$. 
\end{enumerate}
\paragraph{Case 2} Let $C$ be a triangle with the vertex set $V(C) = \{c_1, c_2, c_3\}$ such that $a_1c_1, a_2c_2\in E(G)$. 
Suppose that $C$ is adjacent to some triangle $B$ with the vertex set $V(B)=\{b_1, b_2, b_3\}$. If $b_ia_3 \notin E(G)$, for all $i\in \{1,2,3\}$, then the analysis from Case 1 can be applied on $B$. Otherwise, consider another triangle $T$ adjacent to $B$ and apply the adjusted analysis from Case 1 on triangle $T$ and its neighbours. \\
According to the analysed cases it follows that the graph $G$ is $\mathcal{S}$.  
\end{proof}

Next, we consider the connected cubic graph that consists of vertex-disjoint triangles and diamonds, and prove the Theorem \ref{thm3}. 

\begin{proof}[of Theorem \ref{thm3}]
Consider the following cases. 
\paragraph{Case 1} Let $d_1 = a_1 \in V(A)$, where $A$ is a triangle with the vertex set $V(A) = \{a_1,a_2,a_3\}$.
Consider the following cases. 
\begin{enumerate}
\item[\textbf{1.i.}] The vertex $a_1$ is adjacent to a diamond $H$ with the vertex set $V(H) = \{h_1, h_2, h_3, h_4\}$, where $E(H)= \{h_1h_2,h_2h_3,h_3h_4,h_4h_1,h_2h_4\}$. Let $a_1h_1 \in E(G)$. 
Consider the following subcases.
\begin{enumerate}
\item[\textbf{\textit{1.i.1.}}\label{thm2_case1_1_i_1}] Vertex $a_2$ is adjacent to a diamond $D_1$ different from $H$ and vertex $a_3$ is adjacent to a diamond $D_2$ (which can be equal to one of the diamonds $D_1$, $H$). \\
Since the number of triangles must be even there exists at least one more triangle, say $X$, with the vertex set $V(X) = \{x_1, x_2, x_3\}$ such that one of the cases 1.i.1.a., 1.i.1.b., 1.i.1.c., 1.i.1.d., 1.i.1.e. and 1.i.1.f. from  Figure \ref{graphs} holds. \\ 
 \begin{figure}[!h]
    \centering
    \subfigure[Case 1.i.1.a.]{\includegraphics[scale=0.7]{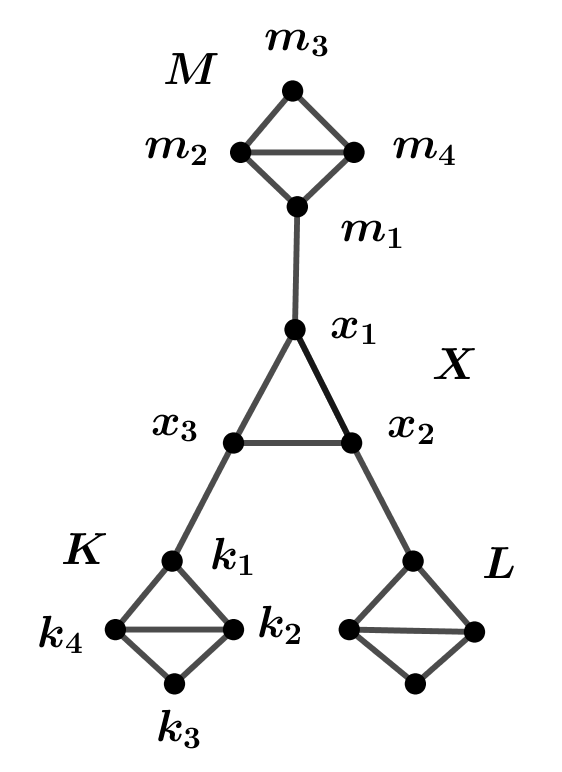}} 
    \hfill
    \subfigure[Case 1.i.1.b.]{\includegraphics[scale=0.7]{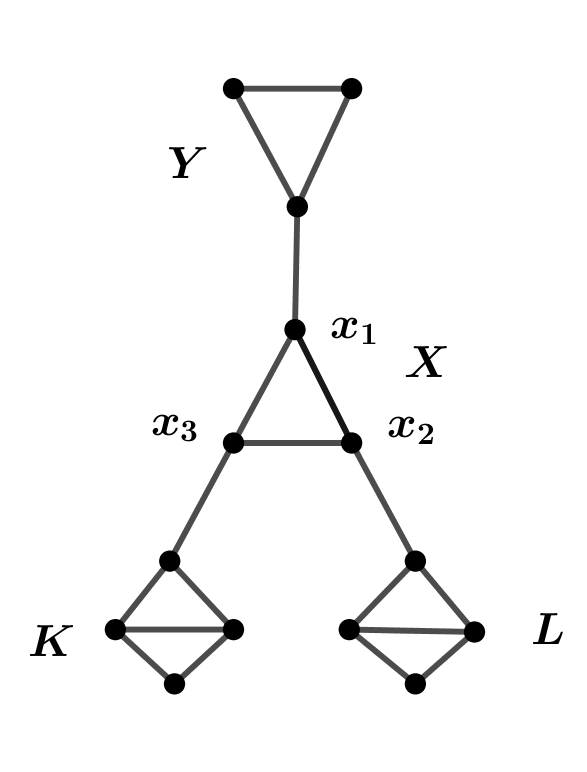}} 
    \hfill
    \subfigure[Case 1.i.1.c.]{\includegraphics[scale=0.7]{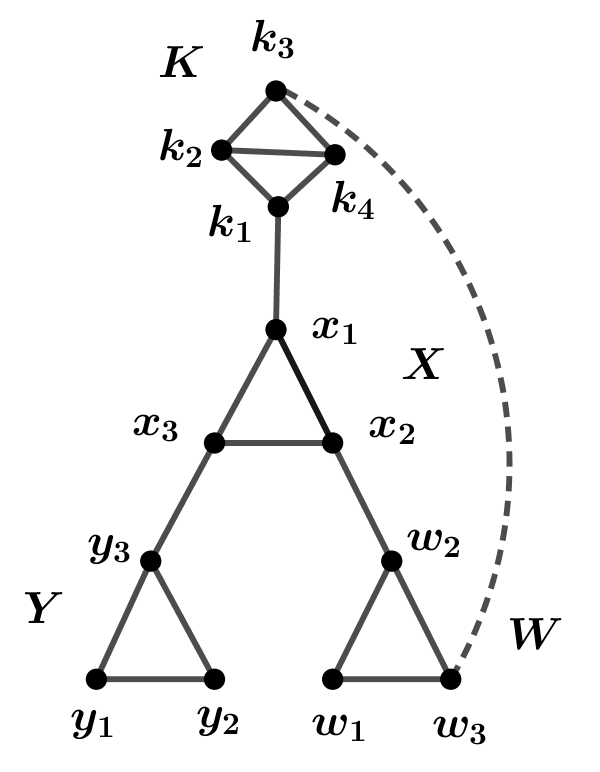}}
   \newline
    \subfigure[Case 1.i.1.d.]{\includegraphics[scale=0.7]{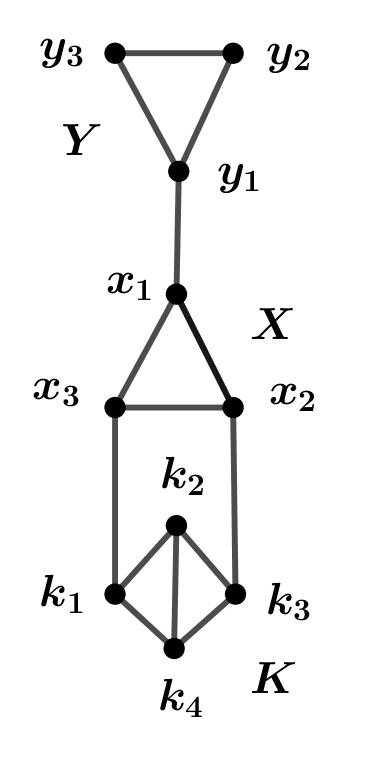}}
     \hfill
    \subfigure[Case 1.i.1.e.]{\includegraphics[scale=0.7]{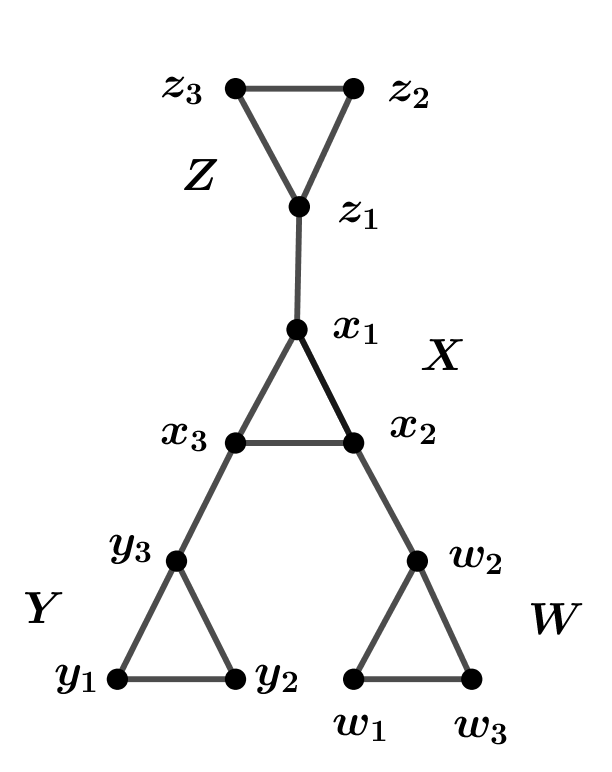}}
     \hfill
    \subfigure[Case 1.i.1.f.]{\includegraphics[scale=0.7]{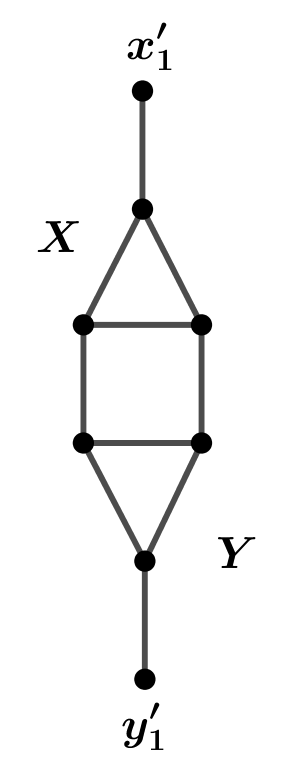}}
    \caption{Subgraphs illustration from Case 1.}
    \label{graphs}
\end{figure}
\noindent For cases 1.i.1.a. and 1.i.1.b, consider $S$-game on the subgraph $G_4$ which vertex set is a union of $V(X)$ and vertex sets of two diamonds that are adjacent to $X$. By Lemma \ref{lema4}, Staller wins. 
\item[1.i.1.c.] In case that one of the triangles $Y$ and $W$ is adjacent to $k_3$, then, w.l.o.g.\ suppose that triangle $W$ is adjacent to $k_3$. 
Then, Staller plays in the following way: \\
$s_1 = k_1$ which forces $d_2 \in V(K)\setminus \{k_1\}$ (a diamond trap), $s_2 = x_2$ which forces $d_3 = x_3$ and $s_3 = y_3$ which forces $d_4 = x_1$. Triangle $Y$ can be adjacent to some diamonds and/or triangles. The vertices of these diamonds and/or triangles together with $V(Y)$ form one of the subgraphs $G_1$, $G_2$ or $G_3$. According to the lemmas  \ref{lema1}, \ref{lema2} and \ref{lema3}, Staller wins. 
\item[1.i.1.d.] Staller plays in the following way: \\
$s_1 = k_1$ which forces $d_2 \in V(K)\setminus \{k_1\}$ (a diamond trap), $s_2 = x_2$ which forces $d_3 = x_1$ and $s_3 = y_1$ which forces $d_4 = x_3$. Triangle $Y$ can be adjacent to some diamonds and/or triangles. Vertices of these diamonds and/or triangles together with $V(Y)$ form one of the subgraphs $G_1$, $G_2$ or $G_3$. According to the lemmas  \ref{lema1}, \ref{lema2} and \ref{lema3}, Staller wins. 
\item[1.i.1.e.] If the triangles $X, Y, Z, W$ are adjacent only to triangles, then we can use the analysis from the proof of Theorem 1.2. Case 1, where $X=B, Z=A$ and $V(Z)\not\subseteq \mathfrak{D}$. Otherwise, if at least one of the triangles are adjacent to diamond then we can have a subgraph from Figure \ref{graphs} (b), (c) or (d) for which we can apply an analysis from the corresponding Case 1.i.1.b, 1.i.1.c or 1.i.1.d.
\item[1.i.1.f.] Consider the $S$-game on the subgraph $G_1$ with the vertex set $V(X) \cup V(Y) \cup \{x_1', y_1'\}$. By Lemma \ref{lema1}, Staller wins.

\item[\textbf{\textit{1.i.2.}} \label{thm2_case1_1_i_2}] The vertices $a_2$ and $a_3$ are adjacent to the same triangle, say $B$, where $a_2b_2, a_3b_3 \in E(G)$. 
\item[1.i.2.a.] $B$ is adjacent to a triangle $Y$ with the vertex set $V(Y) = \{y_1,y_2,y_3\}$. Let $b_1y_1 \in E(G)$. 
Then, there exists at least one more triangle, $X$ such that one
of the cases from Figure \ref{graphs} can occur, so adjusted analysis from Case 1.i.1. can be applied. 
\item[1.i.2.b.] $B$ is adjacent to a diamond $K$ with the vertex set $V(K) = \{k_1,k_2,k_3, k_4\}$, where $E(K)= \{k_1k_2,k_2k_3,k_3k_4,k_4k_1,k_2k_4\}$. Let $b_1k_1 \in E(G)$. 
Then, $s_1=k_1$ forces $d_2=V(K)\setminus \{k_1\}$, $s_2=b_2$ forces $d_3=b_3$, $s_3=a_3$ forces $d_4=b_1$ and $s_4=h_1$ creates a vertex-diamond trap $a_2-H$.
\item[1.i.2.c.] $B$ is adjacent to a diamond $H$, that is, $b_1h_3 \in E(G)$. Since induced subgraph with the vertex set $V(A) \cup V(B) \cup V(H)$ is a connected cubic graph, it follows that this graph is the graph $G$ on 10 vertices.  \\
In her first move Staller plays $s_1 = b_1$. Then
\item[-] if $d_2 = h_1$, Staller plays $s_2 = b_3$ which forces $d_3 = a_2$. Next, $s_3 = b_2$ and Staller creates a double trap $a_3-h_3$. In her next move she isolates either $b_3$ or $b_1$.
\item[-] if $d_2 \in V(H) \setminus \{h_1\}$, Staller plays $s_2 = a_2$ which forces $d_3 = b_3$. Next, $s_3 = a_3$ and Staller creates a double trap $h_1-b_2$. In her next move she isolates either $a_1$ or $b_3$.
\item[-] if $d_2 = a_2$ (or $d_2 = a_3$), Staller plays $s_2 = b_2$ (or $s_2 = b_3$) which forces $d_3 = a_3$ (or $d_3 = a_2$). Next, $s_3 = h_3$ and Staller creates a vertex-diamond trap $b_3-H$ (or $b_2-H$). 
\item[-] if $d_2 = b_2$ ($d_2 = b_3$), Staller plays $s_2 = a_2$ (or $s_2 = a_3$) which forces $d_3 = b_3$ (or $d_3 = b_2$). Next, $s_3 = h_1$ and Staller creates a vertex-diamond trap $a_3-H$ (or $a_2-H$). \\
It follows that the graph $G$ is $\mathcal{S}$.

\item[\textbf{\textit{1.i.3.}}] The vertex $a_2$ is adjacent to a triangle, say $T$ with the vertex set $V(T)=\{t_1, t_2, t_3\}$ and the vertex $a_3$ is adjacent to a triangle, say $B$ with the vertex set $\{b_1, b_2, b_3\}$. Let $a_2t_2, a_3b_3 \in E(G)$. \\
There exists at least one more triangle, say $X$ with the vertex set $V(X)=\{x_1, x_2, x_3\}$. If $X$ is not adjacent to any of $B$, $T$, then one of the cases from Figure \ref{graphs} can occur, so adjusted analysis from Case 1.i.1 can be applied. \\
If $X$ is adjacent to a triangle $B$ or $T$ and there are two edges between them, then we can use Lemma \ref{lema1}, as we have subgraph $G_1$. \\
Suppose that $X$ is adjacent to at least one of the triangles $B$ and $T$ and there is only one edge between $X$ and that triangle. Let $x_1b_1 \in E(G)$. Staller plays in the following way: $s_1=h_1$ forcing $d_2\in V(H)\setminus \{h_1\}$, $s_2=a_3$ forcing $d_3=a_2$, $s_3=b_2$ forcing $d_4=b_1$ and $s_4=x_1$ forcing $d_5=b_3$. \\
If none of $x_2$ and $x_3$ is not adjacent to $H$ then triangles and/or diamonds adjacent to $X$ together with $V(X)$ can form one of the subgraphs $G_1$, $G_2$ or $G_3$ and according to the lemmas \ref{lema1}, \ref{lema2}, and \ref{lema3}, Staller wins the game on these subgraphs. \\
If $x_2$ is adjacent to $H$ and $x_3$ is adjacent to some other diamond $K$, then we can use subgraph $G_3$ with $V(G_3)=V(X)\cup V(K)$, and according to Lemma \ref{lema3}, Staller wins. \\
Otherwise, suppose that $x_2$ is adjacent to $H$ and $x_3$ is adjacent to triangle, say $Y$ with the vertex set $V(Y)=\{y_1,y_2,y_3\}$, where $x_3y_3\in E(G)$. Then, $s_5=y_3$ forcing $d_6=x_2$. 
Triangle $Y$ can be adjacent to some diamonds and/or triangles. 
The vertices of these triangles and/or diamonds different from $A$ and $H$, which are adjacent to $Y$, together with $V(Y)$ and their neighbours form one of the subgraphs $G_1$, $G_2$ or $G_3$. According to the lemmas \ref{lema1}, \ref{lema2}, or \ref{lema3}, Staller wins the game on these subgraphs.  \\
If $x_3t_3 \in E(G)$, then $s_5=t_3$ forces $d_6=x_2$. If $t_1$ is adjacent to a diamond then consider subgraph $G_3$ and according to Lemma \ref{lema3}, Staller wins. Let $t_1$ be adjacent to a triangle $W$ with the vertex set $V(W) = \{w_1, w_2, w_3\}$ and let $t_1w_1 \in E(G)$. Then Staller plays $s_6=w_1$ which forces $d_7=t_2$. Triangle $W$ can be adjacent to triangles and/or diamonds different from $A,X,H$. Vertices of these triangles and/or diamonds together with $W$ can form one of the subgraphs $G_1, G_2$ or $G_3$. So, according to lemmas \ref{lema1}, \ref{lema2} or \ref{lema3}, Staller wins. 
\item[\textbf{\textit{1.i.4.}}] Let $a_2$ be adjacent to a diamond and $a_3$ to a triangle $B$ with the vertex set $V(B)=\{b_1, b_2, b_3\}$ and let $a_3b_3\in E(G)$. \\
If $b_1$ and $b_2$ are adjacent to two different diamonds, then consider subgraph $G_4$ and according to Lemma \ref{lema4} Staller wins. \\
If $b_1$ and $b_2$ are adjacent to the same diamond $K$ with the vertex set $V(K)=\{k_1,k_2,k_3,k_4\}$ and the edge set $E(K)=\{k_1k_2,k_2k_3,k_3k_4, k_4k_1, k_2k_4\}$. Let $b_1k_1, b_2k_3\in E(G)$.
Then, Staller plays in the following way:  $s_1=h_1$ forcing $d_2\in V(H)\setminus \{h_1\}$, $s_2=a_3$ forcing $d_3=a_2$ and $s_3=b_1$ forcing $d_4=b_2$. By playing $s_4=k_3$ Staller creates $b_3-K$ trap. \\
Otherwise, consider some triangle $X$ and the very similar analysis from previous Case 1.i.3 could be used. 
\end{enumerate}
So, the graph $G$ is $\mathcal{S}$.
\item[\textbf{1.ii.}] The vertex $a_1$ is adjacent to a triangle, say $B$ with the vertex set $V(B) = \{b_1, b_2, b_3\}$. Let $a_1b_1 \in E(G)$. Consider the following subcases. 
\begin{enumerate}
\item[\textbf{\textit{1.ii.1}}] The vertex $b_2$ is adjacent to some diamond, say $H$, and  the vertex $b_3$ is adjacent to a diamond, say $K$. Consider the $S$-game on the subgraph $G_4$ with the vertex set $V(B) \cup V(H) \cup V(K)$. By Lemma \ref{lema4}, Staller wins. So, $G$ is $\mathcal{S}$. 
\item[\textbf{\textit{1.ii.2}}] The vertices $b_2$ and $b_3$ are adjacent to the same diamond, say $H$ with the vertex set $V(H) = \{h_1,h_2,h_3,h_4\}$ and the edge set $E(H)=\{h_1h_2, h_2h_3, h_3h_4, h_4h_1, h_2h_4\}$, where $b_2h_1, b_3h_3 \in E(G)$. 
\item[1.ii.2.a.] If there are no more triangles in the graph $G$, then the vertices $a_2$ and $a_3$ are adjacent to the chain of diamonds.  
We have a graph $\omega $ from Figure \ref{f5}. According to Lemma \ref{lema6}, Dominator wins as the first player. 
\item[1.ii.2.b.] Otherwise, suppose that graph contains more triangles. 
If both $a_2$ and $a_3$ are adjacent to some diamonds which do not form a chain of diamonds, then consider some triangle $X$ which could not be adjacent either to $A$ or $B$. One of the cases from Figure \ref{graphs} has to occur, so adjusted analysis from Case 1.i.1 can be applied.  
\item[1.ii.2.c.] At least one of $a_2$, $a_3$ are adjacent to some triangle. Let $Y$ be a triangle with the vertex set $V(Y)=\{y_1, y_2, y_3\}$ such that $a_2y_2 \in E(G)$. \\
If $Y$ is adjacent to some diamond $K$ with the vertex set $V(K) = \{k_1,k_2,k_3,k_4\}$ and the edge set $E(K)=\{k_1k_2, k_2k_3, k_3k_4, k_4k_1, k_2k_4\}$, so there is at least one edge between $Y$ and $K$, say $y_1k_1 \in E(G)$. Staller plays in the following way: $s_1 = k_1$ forcing $d_2 \in V(K)\setminus \{k_1\}$, $s_2=y_3$ forcing $d_3=y_2$, $s_3=a_2$ forcing $d_4=y_1$, $s_4=b_1$ forcing $d_5=a_3$ and $s_5=h_1$, so Staller creates $b_3-H$ trap. \\
If $y_1$ and $y_3$ are adjacent to the same triangle, then consider subgraph $G_1$ of a given graph, so according to Lemma \ref{lema1}, Staller wins. \\
Otherwise, suppose that $y_1$ is adjacent to a triangle $W$ with the vertex set $V(W)=\{w_1, w_2, w_3\}$ and let $y_1w_1\in E(G)$, and $y_3$ is adjacent to a triangle $T$. 
If there is no edge between $A$ and $W$, or if $T=A$, then Staller plays in the following way: 
$s_1=h_1$ which forces $d_2\in V(H)\setminus \{h_1\}$, $s_2 = b_1$ which forces $d_3 = b_3$, $s_3 = a_2$ which forces $d_4 = a_3$, $s_4 = y_3$ which forces $d_5=y_1$ and $s_5 = w_1$ which forces $d_6=y_2$. Triangle $W$ can be adjacent to some triangles and/or diamonds different from $A$, $B$ and $H$. Vertices of these diamonds and/or triangles together with $V(W)$ form one of the subgraphs $G_1$, $G_2$ or $G_3$. According to the lemmas \ref{lema1}, \ref{lema2} and \ref{lema3}, Staller wins. \\
Otherwise, if there is an edge between $A$ and $W$, then the analysis above can be applied on triangle $T$ and its neighbours, instead of $W$.
\item[\textbf{\textit{1.ii.3}}] The vertex $b_2$ is adjacent to a diamond, say $H$ with the vertex set $V(H) = \{h_1,h_2,h_3,h_4\}$, where $b_2h_1 \in E(G)$ and vertex $b_3$ is adjacent to a triangle, say $W$, with the vertex set $V(W) = \{w_1, w_2, w_3\}$, where $b_3w_3 \in E(G)$. If $W$ is the part of a subgraph $G_1$, that is $W$ is adjacent to one more triangle and there are two edges between them, then by Lemma \ref{lema1}, Staller wins. Otherwise, Staller plays in the following way: \\
$s_1 = h_1$ which forces $d_2 \in V(H)\setminus \{h_1\}$ (a diamond trap), $s_2 = b_1$ which forces $d_3 = b_3$, $s_3 = w_3$ which forces $d_4 = b_2$.
Triangle $W$ can be adjacent to diamonds and/or triangles.
\item[1.ii.3.a.] If these diamonds and/or triangles are different from $A$ and $H$, then vertices of these diamonds and/or triangles together with $V(W)$ form one of the subgraphs $G_2$ or $G_3$. According to the lemmas \ref{lema2} and \ref{lema3}, Staller wins.  
\item[1.ii.3.b.] If at least one of the vertices $w_1, w_2$ is adjacent to $A$, e.g. let $w_2a_2 \in E(G)$, then by playing $s_4 = a_2$ Staller creates a double trap  $a_3-w_1$. In her next move Staller isolates either $a_1$ or $w_2$.  
\item[1.ii.3.c.] If $w_2$ is adjacent to some diamond $K$ with the vertex set $V(K)=\{k_1,k_2,k_3,k_4\}$, where $w_2k_1\in E(G)$, then Staller plays $s_4 =k_1$ and creates a vertex-diamond trap $w_1-K$. 
\item[1.ii.3.d.]  If $w_1h_3\in E(G)$ and $w_2$ is adjacent to some triangle different from $A$, say $R$, with the vertex set $V(R) = \{r_1, r_2, r_3\}$, where $w_2r_2 \in E(G)$. Then, Staller plays $s_4 = r_2$ and forces $d_5 = w_1$. Next, 
\item[-] if there is at least one edge between $R$ and $A$, say $a_3r_3$, then Staller plays $s_5 = a_3$ and creates a double trap $a_2-r_1$. In her next move she isolates $a_1$ or $r_3$. 
\item[-] if there are no edges between $R$ and $A$, then $R$ is adjacent to other diamonds and/or triangles. The vertices of these diamonds and/or triangles together with $V(R)$ form one of the subgraphs $G_1$, $G_2$ or $G_3$. According to the lemmas \ref{lema1}, \ref{lema2} or \ref{lema3}, Staller wins.  \\
So, $G$ is $\mathcal{S}$.
\item[\textbf{\textit{1.ii.4}}] The vertex $b_2$ is adjacent to a triangle, say $W$, with the vertex set $V(W) = \{w_1, w_2, w_3\}$, where $b_2w_2 \in E(G)$ and the vertex $b_3$ is adjacent to a triangle, say $Y$, with the vertex set $V(Y) = \{y_1, y_2, y_3\}$, where $b_3y_3 \in E(G)$. \\
Depending of the type of the neighbours of triangle $W$ we can use the adjusted analysis from the proof of Theorem \ref{thm2} for the Case 1.ii or analysis from the previous Case 1.ii.3. 
So, the graph $G$ is $\mathcal{S}$.
\item[\textbf{\textit{1.iii.}}] The vertex $a_1$ is adjacent with triangle $B$ with the vertex set $V(B)=\{b_1, b_2, b_3\}$ and there are two edges between them, $a_1b_1, a_2b_2\in E(G)$. \\
If there exists a diamond, say $H$ adjacent to $B$ or connected with $B$ by a path of triangles (see Figure \ref{fig4}), then Staller for her first move plays $h_1$, and then follows the strategy illustrated on Figure \ref{fig4}. She creates $a_3-b_3$ trap.
\begin{figure}[h]
    \centering
    \includegraphics[scale=0.8]{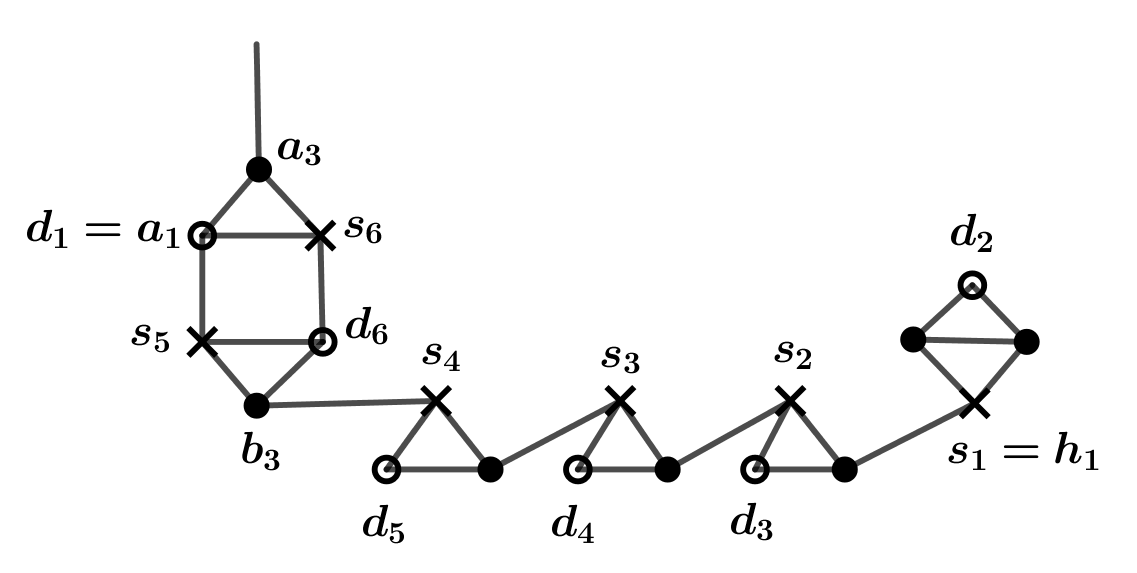}
    \caption{Possible situation for Case 1.iii.}
    \label{fig4}
\end{figure}
\end{enumerate}
\end{enumerate}
Otherwise, the adjusted analysis from Case 2 in the proof of Theorem \ref{thm2} can be applied to prove that Staller wins. 
\paragraph{Case 2} $d_1 \in V(Z)$, where $Z$ is a diamond with the vertex set $V(Z) = \{z_1, z_2, z_3, z_4\}$ and the edge set $E(Z) = \{z_1z_2, z_2z_3, z_3z_4, z_4z_1, z_2z_4\}$. Suppose that $d_1\in \{z_1, z_2\}$. 
Consider the following subcases.
\begin{enumerate}
\item[\textbf{2.i.}] If both $z_1$ and $z_3$ are adjacent to some diamonds then there exists a triangle $X$ adjacent to some triangles or diamonds different from $Z$, such that one of cases from Figure \ref{graphs} can occur. So, adjusted analysis from Case 1.i.1 can be applied. 
\item[\textbf{2.ii.}] Suppose that $z_1$ is adjacent to a diamond and $z_3$ is adjacent to a triangle, say $B$. Adjusted analysis from Case 1.ii can be applied on $B$ to prove that Staller wins. 
\item[\textbf{2.iii.}] Suppose that $z_1$ is adjacent to a triangle $A$ with the vertex set $V(A)=\{a_1, a_2, a_2\}$ and let $z_1a_3 \in E(G)$ and $z_3$ is adjacent to a diamond. \\
If $a_1$ is adjacent to a diamond $H$ and $a_2$ is adjacent to a diamond $K$, then we can consider subgraph $G_4$ on which, by Lemma \ref{lema4}, Staller wins.  \\
If $a_1$ and $a_2$ are adjacent to the same diamond, then we can find some triangle $X$ adjacent to some triangles and/or diamonds different from $A$ and $Z$. So, one case from Figure \ref{graphs} can occur and adjusted analysis from Case 1.i.1 can be applied.  \\
If triangle $A$ is adjacent to some triangle $B$ with the vertex set $V(B)=\{b_1, b_2, b_3\}$ and are two edges between $A$ and $B$ then we can consider subgraph $G_1$ and by Lemma \ref{lema1} Staller wins. Otherwise, if there is only one edge between $A$ and $B$, then, we can use adjusted analysis from Case 1.ii. 
\item[\textbf{2.iv.}] 
The vertex $z_1$ is adjacent to a triangle, say $A$, with the vertex set $V(A) = \{a_1, a_2, a_3\}$, where $z_1a_1 \in E(G)$ and the vertex $z_3$ is adjacent to a triangle, say $B$, with the vertex set $V(B) = \{b_1, b_2, b_3\}$, where $z_3b_1 \in E(G)$. 
If there are two edges between $A$ and $B$, let $a_2b_2, a_3b_3 \in E(G)$. Consider the MBTD game on the subraph $G_1$ with the vertex set $V(A) \cup V(B) \cup \{z_1, z_3\}$. By Lemma \ref{lema1}, Staller wins. So, $G$ is $\mathcal{S}$. \\

Next, suppose that there is one edge between $A$ and $B$, and let $a_3b_3 \in E(G)$.\\
If graph $G$ contains only these two triangles $A$ and $B$, then suppose that triangle
$A$ is adjacent to a diamond, say $K$, with the vertex set $V(K) = \{k_1, k_2, k_3, k_4\}$ and $E(K) = \{k_1k_2, k_2k_3, k_3k_4, k_4k_1, k_2k_4\}$. Let $a_2k_1 \in E(G)$.  We differentiate between the following cases:
\begin{enumerate}
\item[-] If $b_2k_3 \in E(G)$,  Staller plays in the following way: $s_1 = a_1$ which forces Dominator to claim a vertex from $V(K) \cup \{a_3\}$, as otherwise if Staller claims $k_1$ in her second move she will create a vertex-diamond trap $a_3-K$. \\
If $d_2 = a_3$, then $s_2 = b_3$ which forces $d_3 = a_2$, $s_3 = b_2$ which forces $d_4 = z_3$. Next, $s_4 = k_3$ and Staller creates a vertex-diamond trap $b_1-K$. \\
Otherwise, if $d_2 = k_1$ (or $d_2 \in \{k_2, k_3, k_4\}$), then $s_2 = b_3$ which forces $d_3 = a_2$, $s_3 = b_2$ which forces $d_4 = z_3$. Next, $s_4 = b_1$ (or $s_4 = a_3$) and Staller creates a double trap $a_3-k_3$ (or $k_1-b_1$). In her next move Staller isolates either $b_3$ or $b_2$ (or, $a_2$ or $b_3$). So, $G$ is $\mathcal{S}$.
\item[-] Triangle $B$ is not adjacent to a diamond $K$. 
Then, there exists at least one more diamond, say $M$ different from $K$ with the vertex set $V(M) = \{m_1, m_2, m_3, m_4\}$ adjacent to $B$, where $b_2m_1 \in E(G)$. Staller plays in the following way: \\
$s_1 = m_1$ which forces $d_2 \in V(M)\setminus \{m_1\}$ (a diamond trap), $s_2 = b_3$ which forces $d_3 = b_1$ and $s_3 = a_1$ which forces $d_4 = a_2$. Next, $s_4 = k_1$ and Staller creates a vertex-diamond trap $a_3-K$. Dominator can not win. So, $G$ is $\mathcal{S}$.
\end{enumerate}
Otherwise, graph $G$ contains at least four triangles. Consider some triangle $X$ different from $A$ and $B$. One of the cases from Figure \ref{graphs} must hold and adjusted analysis from Case 1.i.1 can be applied. So, $G$ is $\mathcal{S}$. 

Next, suppose there are no edges between $A$ and $B$. \\
If graph $G$ contains only these two triangles $A$ and $B$, then consider the following
\begin{enumerate}
\item[-] Let $K$ with the vertex set $V(K) = \{k_1, k_2, k_3, k_4\}$ be a diamond adjacent to $A$, where $a_2k_1 \in E(G)$, and let $M$ with the vertex set $V(M) = \{m_1, m_2, m_3, m_4\}$ be a diamond adjacent to $B$, where $b_2m_1 \in E(G)$. If $a_3k_3, b_3m_3 \in E(G)$, then we have the graph $\eta $ (see Figure \ref{f4}). By Lemma \ref{lema5}, if $d_1 \in \{z_1, z_3\}$, Dominator wins in the $D$-game. Otherwise, Staller wins. 
\item[-] Otherwise, at least one of the triangles $A, B$ is adjacent to two more diamonds (different from $Z$). Let $K$ and $L$ be two diamonds with the vertex sets $V(K) = \{k_1, k_2, k_3, k_4\}$ and $V(L) = \{l_1, l_2, l_3, l_4\}$, respectively, such that $a_2k_1, a_3l_1\in E(G)$. Staller plays on subgraph $G_4$ with the vertex set $V(A) \cup V(K) \cup V(L)$. By Lemma \ref{lema4}, Staller wins. Statement also holds if $L=M$. So, $G$ is $\mathcal{S}$.
\end{enumerate}
Otherwise, graph $G$ contains at least four triangles. Consider some triangle $X$ different from $A$ and $B$. One of the cases from Figure \ref{graphs} must holds and adjusted analysis from Case 1.i.1 can be applied. So, $G$ is $\mathcal{S}$.
\item[\textbf{2.v.}] 
Vertices $z_1$ and $z_3$ are adjacent to the same triangle, say $A$. Then, we can have situations from Figure \ref{fig1}. 
\end{enumerate}
\begin{figure}[!h]
    \centering
    \subfigure[Case 2.v.a.]{\includegraphics[scale=0.8]{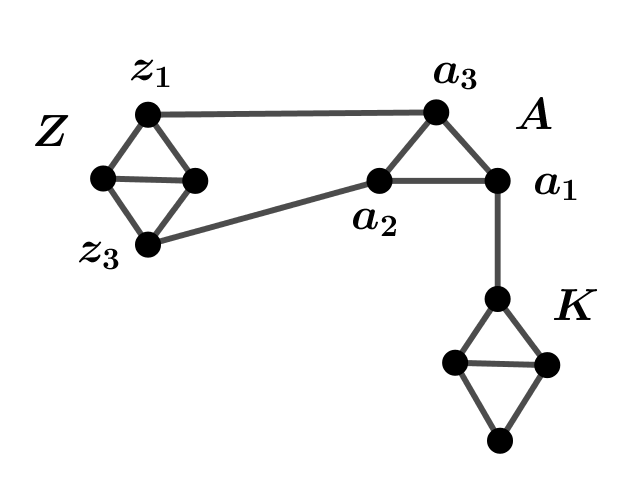}} 
    \hfill
    \subfigure[Case 2.v.b.]{\includegraphics[scale=0.8]{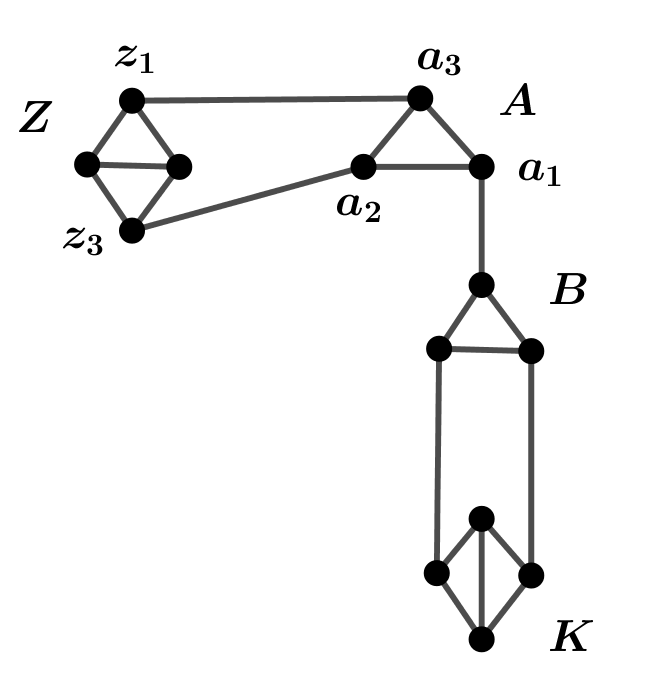}} 
    \hfill
    \subfigure[Case 2.v.c.]{\includegraphics[scale=0.8]{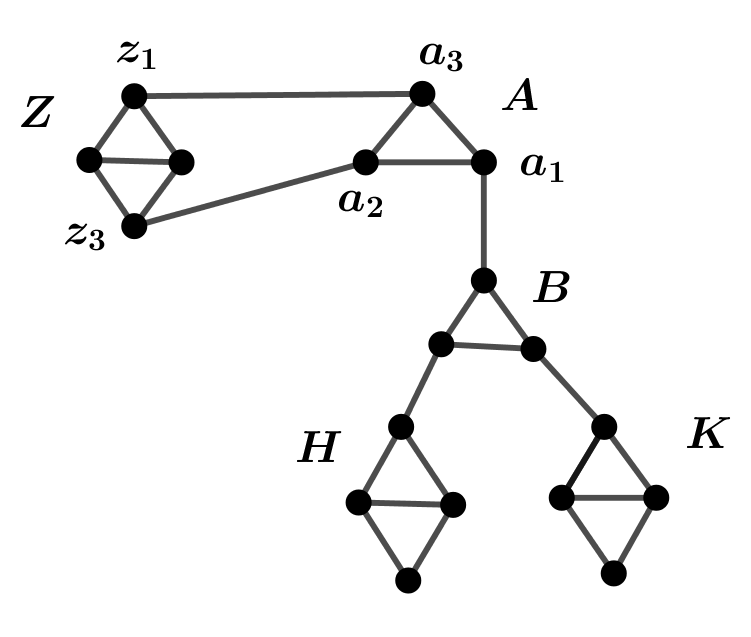}} 
     \hfill
    \subfigure[Case 2.v.d.]{\includegraphics[scale=0.8]{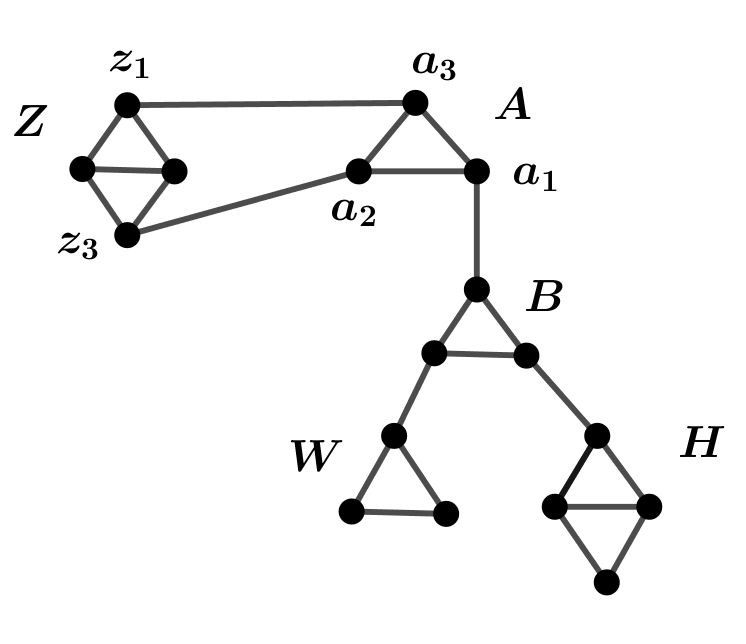}} 
     \hfill
    \subfigure[Case 2.v.e.]{\includegraphics[scale=0.8]{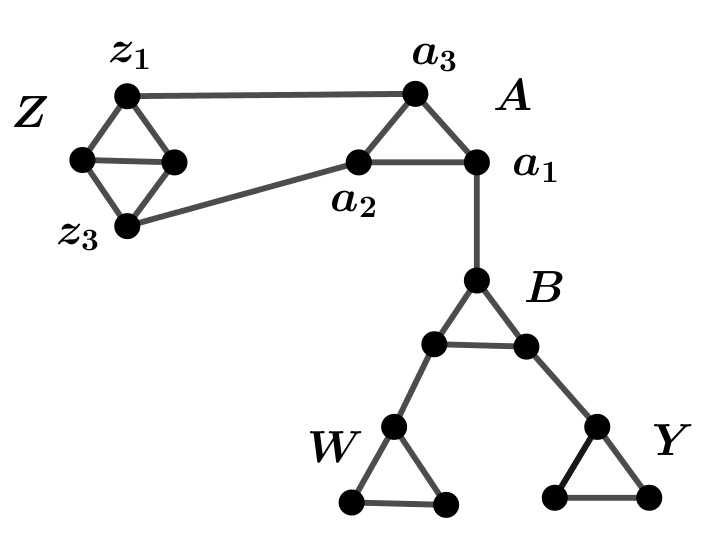}} 
    \caption{Possible situations for Case 2.v.} 
    \label{fig1}
\end{figure}
If we have Case 2.v.(a), then there exists a triangle $X$ such that one of the cases from Figure \ref{graphs} must holds and adjusted analysis from Case 1.i.1 can be applied. \\
If it is Case 2.v.(b), then according to Lemma \ref{lema6}, since $d_1 \in \{z_1, z_2\}$, Staller wins. \\
For Case 2.v.(c) consider subgraph $G_4$ on $V(B) \cup V(H) \cup V(K)$. By Lemma \ref{lema4}, Staller wins. 
For Case 2.v.(d) we can apply adjusted analysis from Case 1.ii.3. \\
For Case 2.v.(e) we can apply adjusted analysis from Case 1.ii.4. 
\end{proof}

Before we give the proof for Theorem~\ref{thm4}, we give the winning strategy for Dominator in the MBTD game on $\mathcal{G}_1 = GP(n,1)$, where $n\geq 3$. Note that it is already proven by~\cite{GHIK19} that $GP(n,1)$ is $\cD$ (precisely, the authors considered the prism $P_2\square C_n$, which is isomorphic to $GP(n,1)$). Here we give a shorter proof for that.
\begin{claim}
MBTD game on $GP(n,1)$, $n\geq 3$ is $\cD$.
\end{claim}
\begin{proof}
Let $V(\mathcal{G}_1) = \{u_1,u_2,...,u_n, v_1,v_2,...,v_n\}$ and let $E(\mathcal{G}_1) = \{u_{i-1}u_i, v_{i-1}v_i | i\in \{2,...,n\}\} \cup \{u_1u_n,v_1v_n\} \cup \{u_iv_i | i \in \{1,...,n\}\}$. \\
If $n$ is even, then $V(\mathcal{G}_1)$ can be partitioned into $4$-sets, each inducing a $C_4$. So, by Proposition \ref{prop}, Dominator wins. \\
Let $n$ be odd. Then, Dominator uses the pairing strategy, where the pairs are $(u_i,v_{i-1})$ for $i\in \{2,...,n\}$ and $(u_1,v_n)$. We need to prove that this is his winning strategy. Suppose that at some point of the game we have a situation that Staller's set contains vertices $u_i, u_{i+2}, v_{i+1}$. This means that vertex $u_{i+1}$ stays uncovered by Dominator. 
This is not possible, because when Staller claimed vertex $u_{i+2}$ (or $v_{i+1}$), Dominator, according to his strategy, must claim vertex $v_{i+1}$ (or $u_{i+2}$) and in this way he covers vertex $u_{i+1}$. A contradiction. 
\end{proof}

\begin{proof}[of Theorem \ref{thm4}]
Consider Generalized Petersen graph $\mathcal{G}_2 = GP(n,2)$. \\
Let $n=6$. Suppose that in his first move Dominator claims some vertex $l$ which belongs to internal polygon (see Figure \ref{GPG2a}).
Staller responds with $s_1 = u$. We consider the following cases:
\begin{enumerate}
\item[Case 1.] $d_2 \in \{z,t,v,w,y\}$. \\
Then, $s_2 = t_2$ which forces $d_3 = t_1$. By playing $s_3 = r_1$ Staller creates a double trap $r-r_2$. In her next move Staller isolates either $r_2$ or $r$.  
\item[Case 2.] $d_2 \in \{t_1, t_2 \}$. \\
Then, $s_2 = z$ which forces $d_3 = w$. By playing $s_3 = r_2$ Staller creates a double trap $r-r_1$. By claiming $r$ or $r_1$ in her next move, Staller will isolate either $r_1$ or $r$.
\item[Case 3.] $d_2 \in \{r, r_1, r_2\}$. \\
Then, $s_2 = z$ which forces $d_3 = w$. By playing $s_3 = t_2$ Staller creates a double trap $y-t_1$. In her next move Staller isolates either $l$ or $t$. 
\end{enumerate}
Next, suppose that Dominator plays his first move on the external polygon. Let $d_1 = y$. 
Then, Staller responds with $s_1 = u$. If $d_2 \in \{z,t,v,w,l\}$ or $d_2 \in \{t_1, t_2\}$, then Staller can use the same strategy as in Case 1 or Case 2, respectively. \\
Otherwise, if $d_2 \in \{r, r_1, r_2\}$, then Staller plays in the following way: $s_2 = t_1$ which forces $d_3 = t_2$. Next, by playing $s_3 = w$, Staller creates $l-z$ trap. By claiming $z$ or $l$ in her next move Staller will isolate $v$ or $y$.  \\ 

Let $n=7$. Due to symmetries of the graph, the vertex $l$ can be the first Dominator's move (see Figure \ref{GPG2b}). 
Staller responds with $s_1 = u$. We consider the following cases:
\begin{enumerate}
\item[Case 1.] $d_2 \in \{t, r, w, z, y_2\}$. \\
Then, $s_2 = t_1$ which forces $d_3 = t_2$ and $s_3 = r_1$ which forces $d_4 = r_2$. Next, $s_4 = v$ and Staller creates a double trap $y_3-y_1$. In her next move Staller isolates either $w$ or $z$.
\item[Case 2.] $d_2 \in \{v, t_1, t_2, y_3\}$. \\
Then, $s_2 = w$ which forces $d_3 = z$. By playing $s_3 = r_2$ Staller creates a double trap $r_1-y_1$. In her next move Staller isolates either $r$ or $y_3$. 
\item[Case 3.] $d_2 \in \{r_1, r_2, y_1\}$. \\
Then, $s_2 = w$ which forces $d_3 = z$, $s_3 = t$ which forces $d_4 = y_2$ and $s_4 = r$ which forces $d_5 = v$. Next, $s_5 = t_2$ and Staller creates a double trap $y_3-t_1$. In her next move Staller isolates either $r_2$ or $t$. 
\end{enumerate}
Let $n=8$ and suppose that in his first move Dominator claims a vertex $l$ which belongs to internal polygon (see Figure \ref{GPG2c}).
Staller responds with $s_1 = u$. We consider the following cases: 
\begin{enumerate}
\item[Case 1.] $d_2 \in \{v, t_1, t_2, y_5\}$. \\
Then, $s_2 = r_2$ which forces $d_3 = r_1$ and $s_3 = t$ which forces $d_4 = y_3$. Next, by playing $s_4 = w$ Staller creates a double trap $y_4-z$. 
By claiming $z$ or $y_4$ in her fifth move, Staller isolates either $v$ or $t_1$. 
\item[Case 2.] $d_2 \in \{w, t, y_3, y_4, z\}$. \\
Then, $s_2 = t_1$ which forces $d_3 = t_2$ and $s_3 = r_1$ which forces $d_4 = r_2$. Next, by playing $s_4 = v$ Staller creates a double trap $y_1-y_5$. By claiming $y_5$ or $y_1$ in her fifth move, Staller isolates either $w$ or $z$.
\item[Case 3.] $d_2 \in \{r_1, r_2\}$. \\
Then, $s_2 = t_2$ which forces $d_3 = t_1$, $s_3 = r$ which forces $d_4 = y_2$ and $s_4 = t$ which forces $d_5 = v$. Next, by playing $s_5 = w$ Staller creates a double trap $y_4-z$. 
By claiming $z$ or $y_4$ in her sixth move, Staller isolates $v$ or $t_1$.
\item[Case 4.] $d_2 \in \{r, y_1, y_2\}$. \\
Then, $s_2 = r_2$ which forces $d_3 = r_1$ and $s_3 = t$ which forces $d_4 = y_3$. Next, by playing $s_4 = w$, Staller creates a double trap $y_4-z$. 
By claiming $z$ or $y_4$ in her fifth move, Staller isolates $v$ or $t_1$.
\end{enumerate}
Next, suppose that Dominator plays his first move on the external polygon. Let $d_1 = y_2$. 
Then, Staller responds with $s_1 = u$. If $d_2 \in \{v, t_1, t_2, y_5\}$ or if $d_2 \in \{w, t, y_3, y_4, z\}$, then Staller can use the same strategy as in Case 1 or Case 2, respectively. If $d_2 \in \{r_1, r_2\}$, then Staller plays in the following way: $s_2 = w$ which forces $d_3 = z$, $s_3 = y_4$ which forces $d_4 = t$. Next, by playing $s_4 = t_2$ Staller creates $l-t_1$ trap. By playing $l$ or $t_1$ in her next move Staller will isolate either $y_3$ or $t$. Finally, if $d_2 \in \{r, y_1, l\}$, then Staller can use the same strategy as in Case 4. 

\begin{figure}[!h]
    \centering
    \subfigure[$GP(6,2)$ \label{GPG2a}]{\includegraphics[scale=0.7]{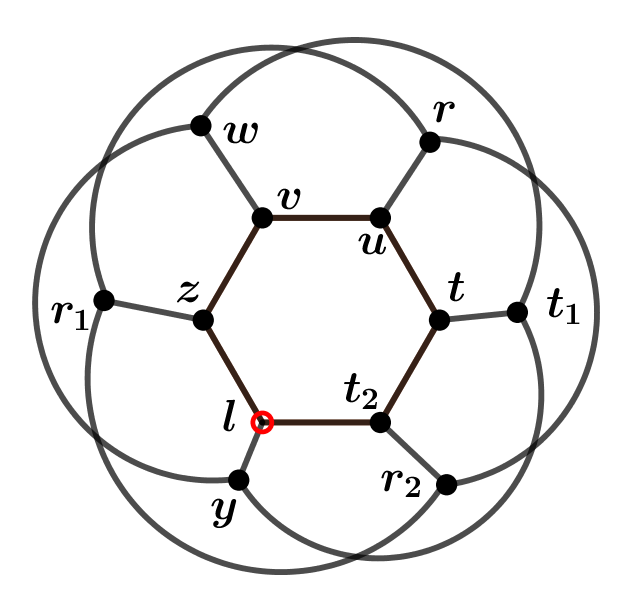}} 
    \hspace*{0.5cm}
    \subfigure[$GP(7,2)$ \label{GPG2b}]{\includegraphics[scale = 0.65]{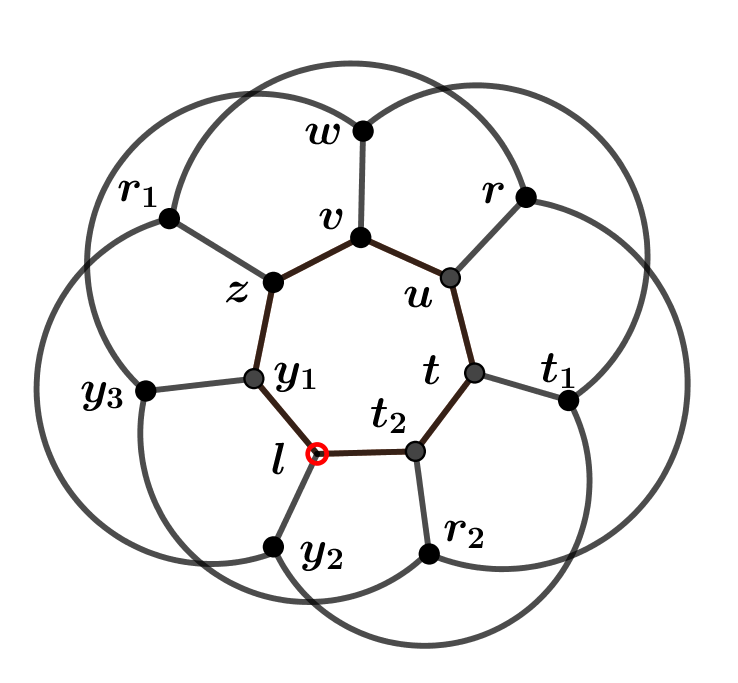}} 
    \hspace*{0.5cm}
    \subfigure[$GP(8,2)$ \label{GPG2c}]{\includegraphics[scale = 0.5]{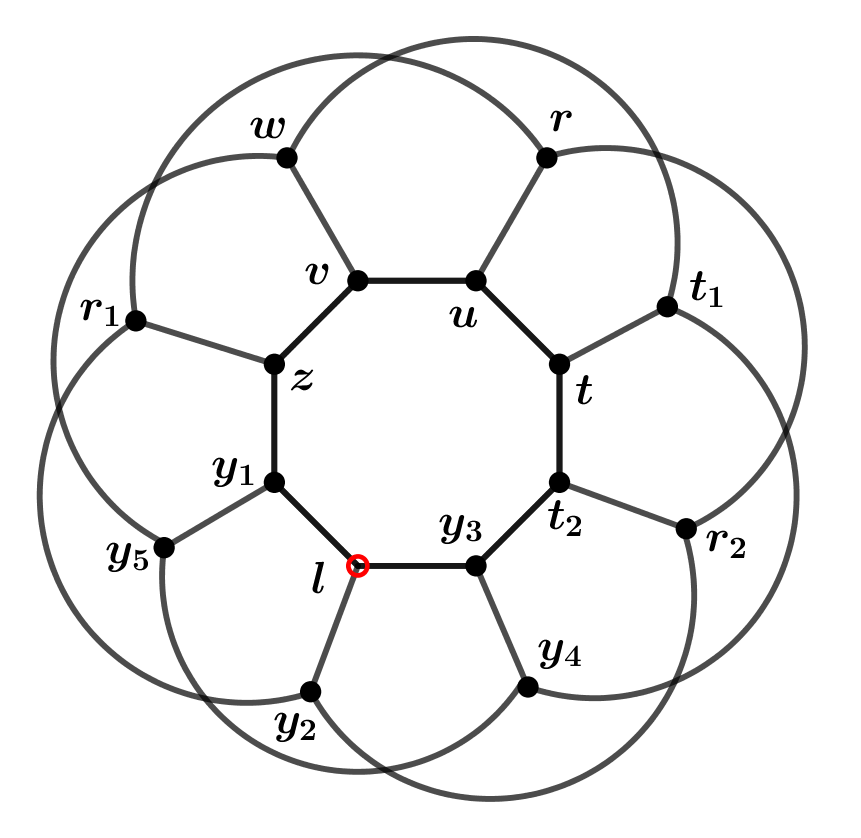}} 
    \caption{Generalized Petersen graph.}
    \label{GPG2}
\end{figure}
\noindent Consider Generalized Petersen graph $\mathcal{G}_2 = GP(n,2)$, where $n\geq 9$. \\
After Dominator's first move, Staller can find subgraph  $\tau \subseteq \mathcal{G}_2$ such that $d_1 \notin V(\tau)$. 
Suppose that in his first move Dominator claims some vertex $l$ which belongs to internal polygon (see Figure \ref{GPGa}). Consider subgraph $\tau \subseteq \mathcal{G}_2$ with the vertex set $\{u,v,w,z,t,t_1,t_2,r_1,r_2,y_1,y_2,y_3,y_4,y_5,y_6\}$, where the vertices $u$ and $v$ are at distance 4 from the vertex $l$ on the internal polygon. The subgraph $\tau $ is illustrated in Figure \ref{GPGb}. 
In her first move Staller claims $s_1 = u$. It is enough to consider the cases when $d_2 \in \{v,w,t,t_1,t_2, y_4, y_5, y_6\}$.
\begin{enumerate}
\item[Case 1.] $d_2 \in \{w, t, y_4, y_5 \}$. \\
Then, $s_2 = t_1$ which forces $d_3 = t_2$ and $s_3 = r_1$ which forces $d_4 = r_2$. Next, by playing $s_4 = v$ Staller creates a double trap $y_1-y_6$. 
In her next move Staller isolates either $w$ or $z$ by claiming $y_6$ or $y_1$. 
\item[Case 2.] $d_2 = \{v, t_1, t_2, y_6\}$. \\
Then, $s_2 = r_2$ which forces $d_3 = r_1$ and $s_3 = t$ which forces $d_4 = y_4$. Next, by playing $s_4 = w$ Staller creates a double trap $y_5-z$. 
In her next move Staller isolates either $t_1$ or $v$ by claiming $y_5$ or $z$. 
\end{enumerate}
\begin{figure}[!h]
    \centering
    \subfigure[$GP(9,2)$ \label{GPGa}]{\includegraphics[scale=0.9]{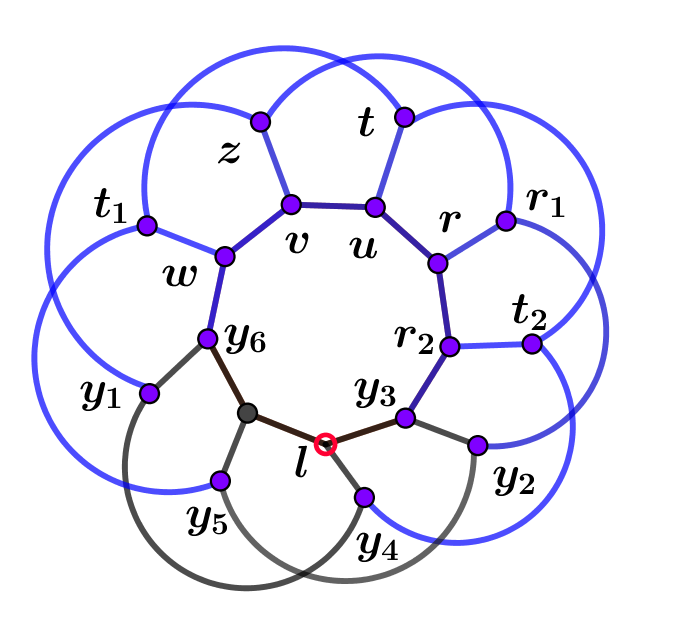}} 
    \hfill
    \subfigure[$\tau $ \label{GPGb}]{\includegraphics[scale = 0.7]{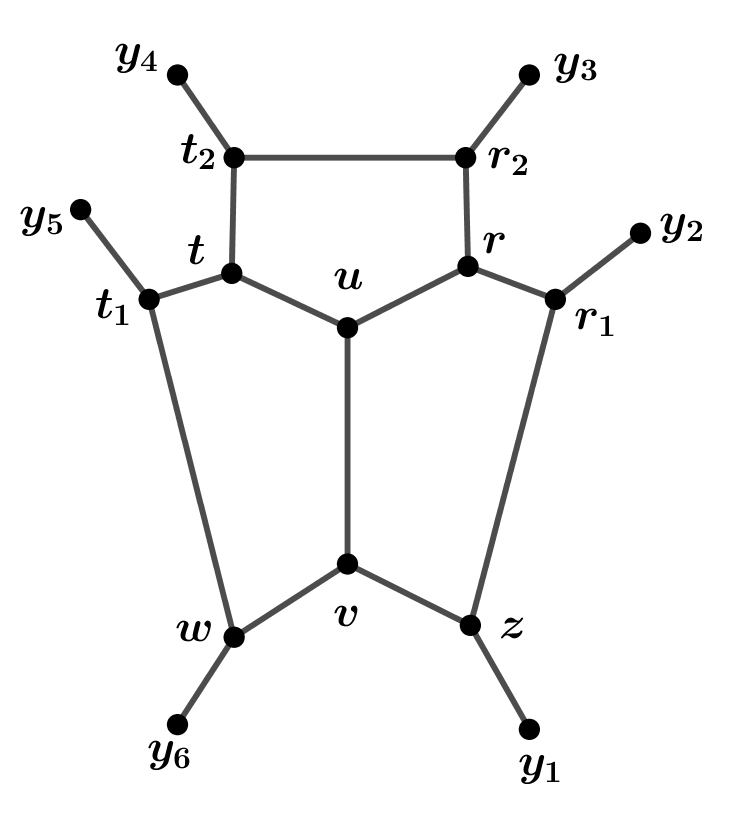}} 
    \caption{(a) Generalized Petersen graph $GP(9,2)$ and (b) its subgraph $\tau $.}
    \label{GPG}
\end{figure}
\end{proof}

In the following we prove Theorem \ref{thm5}.

\begin{proof}[of Theorem \ref{thm5}]
Consider the cubic bipartite graph on $n$ vertices with the vertex set
$\{u_1,...,u_{n/2},v_1,...,v_{n/2}\}$. Let $U = \{u_1,...,u_{n/2}\}$ and $V = \{v_1,...,v_{n/2}\}$ be a bipartition of the graph.
Add an edge from each $u_i$ to $v_i$, $v_{i+i}$ and $v_{i+2}$ (with indices modulo $n/2$). \\ 

\noindent W.l.o.g.\ suppose that $s_1 = u_i \in U$, for some $i\in \{1,2,...,n/2\}$. Then $d_1 = u_{i-1} \in U$, modulo $n/2$.  
Note that every two vertices $u_{i-1}$ and $u_i$ from $U$ have two common neighbours in $V$, $v_i$ and $v_{i+1}$ (and every two $v_{i-1}, v_i \in V$ have two common neighbours in $U$, $u_{i-1}, u_{i-2}$). 
In every other round $r\geq 2$, Dominator plays in the following way. 
If Staller claims a vertex which is a common neighbour of two vertices, say  $u_{k-1}$ and $u_k$ such that for example, $u_{k-1} \in \mathfrak{D}$ and $u_k \in \mathfrak{S}$, then Dominator responds by claiming the other common neighbour of these two vertices. Otherwise, if Staller claimed some vertex $u_l$ (or $v_l$) which is not common neighbour of any two vertices, $x,y \in U$ or $V$, such that, for example, $x\in \mathfrak{D}$ and $y\in \mathfrak{S}$  (or vice versa), then Dominator claims a free vertex $u_{l-1}$ or $u_{l+1}$, with preference $u_{l-1}$ (or, $v_{l-1}$ or $v_{l+1}$ with preference $v_{l-1}$) modulo $n/2$. If Dominator can not find such a free vertex, he claims an arbitrary free vertex from the graph with the preference that a vertex is a neighbour of vertex which is claimed by him earlier in the game. \\ \\
We prove that this is a winning strategy for Dominator. 
Suppose that $v_i,v_{i+1}, v_{i+2} \in \mathfrak{S}$ for some $i\in \{1,2...,n/2\}$ modulo $n/2$, that is, Staller isolated vertex $u_i$. 
This means that when Staller claimed $v_i$, Dominator responded with $v_{i-1}$, but then when Staller claimed $v_{i+1}$ (or $v_{i+2}$), according to his strategy, Dominator had to take $v_{i+2}$ (or $v_{i+1}$). A contradiction. 
\end{proof}

Finally, we consider MBTD game on the connected cubic graph which is disjoint union of claws and prove Theorem \ref{thm6}.

\begin{proof}[of Theorem \ref{thm6}]
The graph $G$ is a connected cubic graph on $4k$ vertices formed with $k\geq 2$ disjoint claws $\mathcal{C}_i$, $i\in \{1,...,k\}$. Let $V(\mathcal{C}_i) = \{x_i, y_i,z_i,t_i\}$, where $t_i$ is a center of $\mathcal{C}_i$, for every $i\in \{1,...,k\}$. \\
First, suppose that $k=2$. Let $E(G) = E(G[\mathcal{C}_1])\cup E(G[\mathcal{C}_2]) \cup \{x_1x_2, y_1y_2, z_1z_2, x_1y_2, y_1z_2, z_1x_2 \}$. \\
The graph can be partitioned into two 4-sets, $\{x_1,t_1,y_1,y_2\}$ and $\{z_1, z_2, t_2, x_2\}$ each inducing a $C_4$ (see Figure \ref{clawsa}). By Proposition \ref{prop} and Proposition \ref{cor2}, Dominator wins. \\ 
Let $k=3$. Let $E(G) = E(G[\mathcal{C}_1])\cup E(G[\mathcal{C}_2]) \cup E(G[\mathcal{C}_3]) \cup \{x_{i-1}x_i, y_{i-1}y_i, z_{i-1}z_i | i \in \{2,3\}\} \cup \{x_1x_3,y_1y_3,z_1z_3\}$. \\
It is enough to consider the case when $d_1 \in V(\mathcal{C}_1)$. The cases when $d_1 \in V(\mathcal{C}_2)$ or $d_1 \in V(\mathcal{C}_3)$ are symmetric. 
\begin{enumerate}
\item[Case 1.] $d_1 = x_1$. \\
Then, $s_1 = t_1$. After Dominator's second move either all vertices from $\mathcal{C}_2$ are free or all vertices from $\mathcal{C}_3$ are free. Suppose that all vertices from $\mathcal{C}_3$ are free. Also, at least two of the vertices $x_2, y_2, z_2$ must be free. Suppose that $x_2$ and $z_2$ are free. Then, $s_2 = z_3$ which forces $d_3 = z_2$. By $s_3 = x_3$ Staller creates a double trap $x_2-y_3$. In her next move Staller isolates either $x_1$ or $t_3$. \\
If $d_1 \in \{y_1, z_1\}$, the proof is very similar.
\item[Case 2.] $d_1 = t_1$. \\
Then, $s_1 = t_2$. 
\begin{enumerate}
\item[Case 2.1.] $d_2 = t_3$. Then, $s_2 = z_3$ which forces $d_3 = z_1$. By playing $s_3 = x_3$, Staller creates a double trap $x_1-y_3$. In her fourth move Staller isolates either $x_2$ or $t_3$. 
\item[Case 2.2.] $d_2 \in \{x_i, y_i, z_i\}$, $i\in \{1,2,3\}$. \\
Let $d_2 = x_i$. If $i=1$, then Staller will make her next move on $\mathcal{C}_3$ and she will force Dominator to play his next move on $\mathcal{C}_1$, if $i=3$, Staller will make her next move on $\mathcal{C}_1$ and force Dominator to play on $\mathcal{C}_3$. If $i=2$, then she can make her next move either on $\mathcal{C}_1$ or $\mathcal{C}_3$. \\
Suppose that $d_2 = x_1$. Then, 
$s_2 = z_3$ which forces $d_3 = z_1$. By playing $s_3 = y_3$ Staller creates a double trap $y_1-x_3$. In her fourth move Staller isolates either $y_2$ or $t_3$. \\
The proof is very similar if $d_2 = y_i$ or $d_2 = z_i$, $i\in \{1,2,3\}$.
\end{enumerate}
\end{enumerate}
\begin{figure}[!h]
    \centering
    \subfigure[Two claws \label{clawsa}]{\includegraphics[scale=0.8]{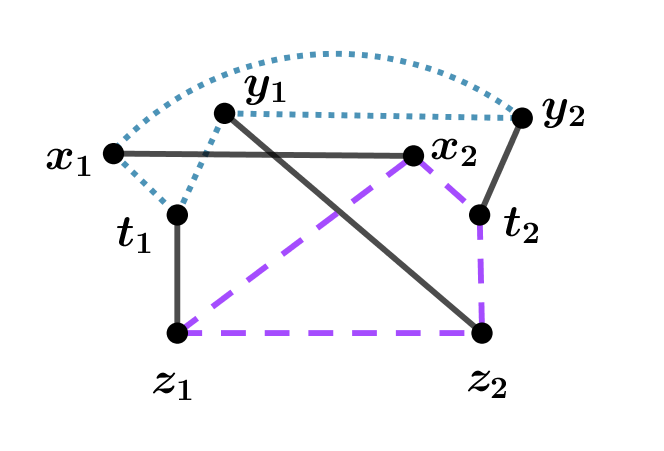}} 
    \hfill
    \subfigure[Three consecutive claws \label{clawsb}]{\includegraphics[scale = 0.8]{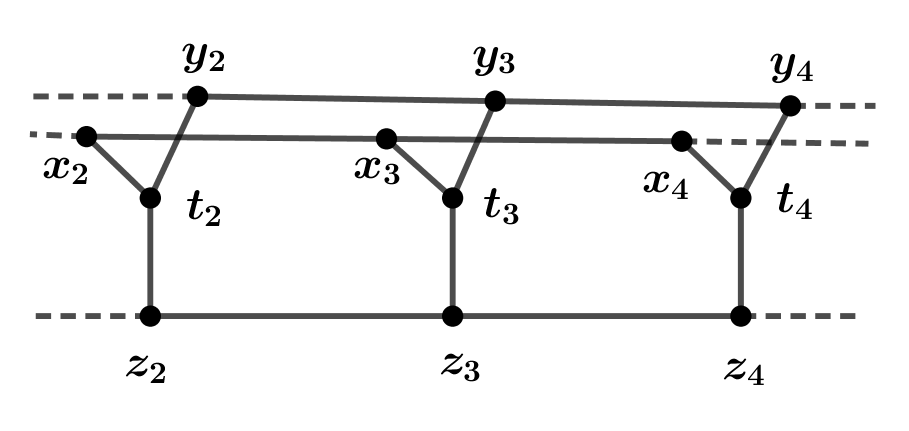}} 
    \caption{Examples of graphs formed of claws.}
\label{claws}
\end{figure}
Let $k\geq 4$. After Dominator's first move, Staller can find three consecutive claws $\mathcal{C}_{i-1}$, $\mathcal{C}_i$  $\mathcal{C}_{i+1}$ such that all vertices from these three claws are free. 
Suppose that these three claws are $\mathcal{C}_2, \mathcal{C}_3$ and $\mathcal{C}_4$. 
Staller will play on a subgraph with the vertex set $V(\mathcal{C}_2), V(\mathcal{C}_3) \cup V(\mathcal{C}_4)$, and the edge set $E(G[\mathcal{C}_2])\cup E(G[\mathcal{C}_3]) \cup E(G[\mathcal{C}_4]) \cup \{x_{i-1}x_i, y_{i-1}y_i, z_{i-1}z_i | i \in \{3,4\}\}$ (Figure \ref{clawsb}). In her first move Staller plays $s_1 = t_3$. \\
If $d_2 \in V(\mathcal{C}_2)$, then Staller will make her next move on $\mathcal{C}_4$ and she will force Dominator to play his next move on $\mathcal{C}_2$, if $d_2 \in V(\mathcal{C}_4)$, Staller will make her next move on $\mathcal{C}_2$ and force Dominator to play on $\mathcal{C}_4$. If $d_2 \in V(\mathcal{C}_3)$, then she can make her moves either on $\mathcal{C}_2$ or $\mathcal{C}_4$. \\
Suppose that $d_2 \in V(\mathcal{C}_2) \cup V(\mathcal{C}_3)$.  \\
Let $d_2 = x_2$ or $d_2 = t_2$. Then, $s_2 = z_4$ which forces $d_3 = z_2$. By playing
$s_3 = y_4$ Staller creates a double trap $y_2-x_4$. In her next move she isolates either $y_3$ or $t_4$.  \\
The cases when $d_2 \in \{y_2, z_2\}$ are symmetric. 
If $d_2 \in \{x_3, y_3, z_3\}$, Staller can apply the same strategy. 
\end{proof}

\begin{remark}
Note that if in the $D$-game on the connected cubic graph $G$ on $n\geq 6$ vertices after Dominator's first move Staller can find at least one of the subgraphs $G_1$, $G_4$, $\tau $, or subgraph which consists of three consecutive connected claws as in Figure \ref{claws}(b), such that all vertices from that subgraph are free, then the graph $G$ is $\mathcal{S}$. 
\end{remark}

\section{Concluding remarks}
In this paper we considered several types of connected cubic graphs in MBTD game and determined which are $\mathcal{D}$ and which are $\mathcal{S}$. In order to determine the outcome of the game, we have focused on finding a representative subgraph of the given graph. 
As we can see from this paper, finding a suitable subgraph makes it easier to determine the winner of the game and helps in characterization of cubic connected graphs. However, we have not covered all connected cubic graphs, so there are still open problems related to this topic. Therefore, it would be interesting to find some other subgraphs that could contribute to expanding the class of cubic connected graphs for which the winner is known in MBTD game.

\textbf{Biased games.} We are curious to know what will happen in the \emph{biased} setup of MBTD game. Given two positive integers, $a$ and $b$, representing the \emph{biases} of Staller and Dominator, respectively, in the biased $(a:b)$ MBTD game, Staller claims exactly $a$ and Dominator claims exactly $b$ elements of the board in each move. Now, if the biases of the players are the same, i.e.\ \emph{fair $(a:a)$ game}, for $a\geq 2$, we wonder whether the outcome of the games change compared to the outcome of the $(1:1)$ MBTD games that were previously studied.

Finally, we wonder how the situation changes if biased non-fair $(a:b)$ MBTD games are played, i.e.\ the games in which $a\neq b$.

\textbf{Biased fractional domination games.} 
Given a graph $G$, a real-valued function $f:V(G)\rightarrow [0,1]$ is a fractional dominating function if $\sum_{u\in N[v]}f(v) \geq 1$ holds for every vertex $v$ and its closed neighborhood $N[v]$ in $G$. The aim of the game is to minimize $\sum_{u\in N[v]}f(v)$. 
In the fractional domination game, introduced by~\cite{BT19}, one player's move consists of a (possibly infinite) sequence $(v_{i_1}, w_1), (v_{i_2}, w_2),...$, where $v_{i_1}, v_{i_2},...$ are vertices of graph $G$ and $w_1, w_2,..$ are real numbers from $(0,1]$ assigned to vertices and after player's move it is required that $\sum_{k\geq 1} w_k = 1$.

In the biased $(a:b)$ fractional total domination game on graph $G$, one round would consists of $a$ moves played by the first player followed by $b$ moves played by the second player such that after each round $\sum_{k\geq 1} w_k = a$ is required for the first player and $\sum_{k\geq 1} w_k = b$ is required for the second player.
So, it would be interesting to consider biased Maker--Breaker fractional total domination game adjusted to these rules and find the optimal winning strategies of players.

\acknowledgements
\label{sec:ack}
The authors would like to thank the anonymous referee for the valuable comments.

\nocite{*}
\bibliographystyle{abbrvnat}
\bibliography{bibfile-dmtcs}
\label{sec:biblio}

\end{document}